\documentclass[11pt]{article}
\usepackage{amssymb}
\PassOptionsToPackage{obeyspaces}{url}
\usepackage[hidelinks]{hyperref}
\usepackage{bookmark}
\bookmarksetup{open,numbered,depth=5}
\usepackage{amsfonts}
\usepackage{amsmath}
\usepackage{dsfont} %indicator function
\usepackage{amsthm,enumitem}
\usepackage{pgf,tikz}
\usepackage{bm}
\usepackage{cancel}
\usepackage{comment}
\usetikzlibrary{calc}
\usepackage[margin = 2.2cm]{geometry}

\usepackage{etoolbox} 
\patchcmd{\thebibliography}{\leftmargin\labelwidth}{\leftmargin\labelwidth\addtolength\itemsep{-0.1\baselineskip}}{}{}

\usepackage{mathtools}
\usepackage{xstring}

\author{
Zichao Dong\thanks{Extremal Combinatorics and Probability Group (ECOPRO), Institute for Basic Science (IBS), Daejeon, South Korea. Supported by the Institute for Basic Science (IBS-R029-C4). \texttt{$\{$zichao,hongliu$\}$@ibs.re.kr}. }
\and Jun Gao\thanks{Mathematics Institute, University of Warwick, Coventry, UK. Supported by ERC Advanced Grant 101020255 and the Institute for Basic Science (IBS-R029-C4). \texttt{gj950211@gmail.com}}
\and Ruonan Li\thanks{School of Mathematics and Statistics, Northwestern Polytechnical University, Xi'an, China. Supported by National Natural Science Foundation of China (No.~11901459, No.~12131013), Shaanxi Fundamental Science Research Project for Mathematics and Physics (22JSZ009), and China Scholarship Council (No.~202306290113).
\texttt{rnli@nwpu.edu.cn}. }
\and Hong Liu\footnotemark[1]
}
\date{}

\title{Induced rational exponents and bipartite subgraphs in $K_{s,s}$-free graphs}

\usepackage[nameinlink]{cleveref}

\newtheorem{theorem}{Theorem}[section]
\newtheorem{lemma}[theorem]{Lemma}
\newtheorem{corollary}[theorem]{Corollary}

\newtheorem{proposition}[theorem]{Proposition}

\newtheorem{conjecture}[theorem]{Conjecture}
\newtheorem{claim}[theorem]{Claim}
\newtheorem{remark}[theorem]{Remark}

\crefname{remark}{remark}{remarks}
\crefname{claim}{claim}{claims}

\newenvironment{poc}{\begin{proof}[Proof of Claim]}{\end{proof}}

\newcommand*{\eqdef}{\stackrel{\mbox{\normalfont\tiny def}}{=}} % definition by equality                                      
\newcommand*{\veps}{\varepsilon}                                % Nice-looking epsilon
\newcommand*{\E}{\mathbb{E}}                                    % Expectation
\newcommand*{\N}{\mathbb{N}}                                    % Natural numbers
\newcommand*{\cH}{\mathcal{H}}
\newcommand*{\cF}{\mathcal{F}}
\newcommand*{\cG}{\mathcal{G}}

\newcommand*{\cT}{\mathcal{T}}
\newcommand*{\cP}{\mathcal{P}}
\newcommand*{\cA}{\mathcal{A}}
\newcommand*{\cB}{\mathcal{B}}
\newcommand*{\cC}{\mathcal{C}}
\newcommand*{\cV}{\mathcal{V}}
\newcommand*{\cW}{\mathcal{W}}
\newcommand*{\cX}{\mathcal{X}}
\newcommand*{\cZ}{\mathcal{Z}}
\newcommand*{\vx}{\textbf{\textup{x}}}
\newcommand*{\vy}{\textbf{\textup{y}}}

\DeclareMathOperator{\ex}{ex}
\DeclareMathOperator{\Hom}{Hom}

\allowdisplaybreaks[4]

\begin{document}

\maketitle

\begin{abstract}
%What is the maximum number of edges in an $n$-vertex graph containing neither a copy of the complete bipartite subgraph $K_{s, s}$ nor an induced copy of a given bipartite graph $H$? Motivated by various applications from discrete geometry and from structural graph theory, Hunter, Milojevi\'{c}, Sudakov, and Tomon [JCTB 2025] recently conjectured that such a graph cannot have many more edges than an $n$-vertex $H$-free graph, and provided evidence by studying the cases when $H$ is a tree, an even cycle, the cube graph and bipartite graphs with bounded degree on one side.
In this paper, we study a general phenomenon that many extremal results for bipartite graphs can be transferred to the induced setting when the host graph is $K_{s, s}$-free. As manifestations of this phenomenon, we prove that every rational $\frac{a}{b} \in (1, 2), \, a, b \in \mathbb{N}_+$, can be achieved as Tur\'{a}n exponent of a family of at most $2^a$ induced forbidden bipartite graphs, extending a result of Bukh and Conlon [JEMS 2018]. Our forbidden family is a subfamily of theirs which is substantially smaller. A key ingredient, which is yet another instance of this phenomenon, is supersaturation results for induced trees and cycles in $K_{s, s}$-free graphs. We also provide new evidence to a recent conjecture of Hunter, Milojevi\'{c}, Sudakov, and Tomon [JCTB 2025] by proving optimal bounds for the maximum size of $K_{s, s}$-free graphs without an induced copy of theta graphs or prism graphs, whose Tur\'{a}n exponents were determined by Conlon [BLMS 2019] and by Gao, Janzer, Liu, and Xu [IJM 2025+]. 
\end{abstract}

\section{Introduction}
The \emph{Tur\'{a}n number} of a family of graphs $\cH$, denoted by $\ex(n, \cH)$, refers to the maximum number of edges in an $n$-vertex graph forbidding any graph in $\cH$ as a subgraph. If $\cH=\{H\}$, we instead  write $\ex(n, H)$. The study of Tur\'{a}n numbers dates back to Mantel \cite{mantel} and Tur\'{a}n \cite{turan}, determining $\ex(n, H)$ when $H$ is a clique. In general, the Erd\H{o}s--Stone--Simonovits theorem \cite{erdos_simonovits_1966,erdos_stone} finds sharp asymptotics of $\ex(n, H)$ for all non-bipartite $H$. In the bipartite case, the behavior of $\ex(n, H)$ gets much more mysterious. For instance, the correct asymptotics of $\ex(n, K_{4, 4})$ and $\ex(n, C_8)$ are unknown. We refer to \cite{furedi_simonovits} for an extensive survey on bipartite Tur\'{a}n numbers. 

The \emph{induced Tur\'{a}n number} of a family $\cH$, denoted by $\ex^*(n, \cH, s)$, is the maximum number of edges in an $n$-vertex graph forbidding subgraph $K_{s, s}$ and induced subgraph $H \, (\forall H \in \cH)$. Similarly, we write $\ex^*(n, H, s)$ when $\cH = \{H\}$. When $s$ is sufficiently large, the definition directly implies that $\ex^*(n, H, s) \ge \ex(n, H)$. The study of $\ex^*(n, \cH, s)$ is related to various problems in structural graph theory and discrete geometry. Readers are directed to \cite[Section 1]{hunter_milojevic_sudakov_tomon} for a survey on such connections. 

Very recently, Hunter, Milojevi\'{c}, Sudakov, and Tomon~\cite{hunter_milojevic_sudakov_tomon} conjectured that for every bipartite graph $H$, the Tur\'{a}n number $\ex(n, H)$ and the induced variant $\ex^*(n, H, s)$ cannot be far apart from each other: $\ex^*(n, H, s) =O_{H, s}(\ex(n, H))$. This conjecture is consistent with all known results by now. They~\cite{hunter_milojevic_sudakov_tomon} found upper bounds on $\ex^*(n, H, s)$ matching those current best known bounds on $\ex(n, H)$ when the bipartite graph $H$ has one side degree bounded (see \cite{furedi,alon_krivelevich_sudakov,girao_hunter,bourneuf_bucic_cook_davies}), as well as when $H$ is a tree (see \cite{scott_seymour_spirkl}), an even cycle (see \cite{bondy_simonovits,lazebnik_ustimenko_woldar}), or the cube graph (see \cite{erdos1964,erdos_simonovits}).

We believe that a more general phenomenon holds. That is, extremal results for bipartite graphs (such as Tur\'an problem, supersaturations etc.) can be extended to the induced setting when we impose a locally sparse condition (such as $K_{s,s}$-freeness) on the host graphs. In this paper, we propose the study of the \emph{meta-problem} ``transfering bipartite extremal results to induced setting in locally sparse graphs''. In what follows we discuss in details some positive such instances that we prove.

\subsection{Rational exponents}

A longstanding conjecture of Erd\H{o}s and Simonovits (see \cite{erdos1981} for instance) in extremal graph theory says that every rational number between $1$ and $2$ can be realized as some Tur\'{a}n exponent. That is, for any rational $q \in (1, 2)$, there exists a bipartite $H$ with $\ex(n, H) = \Theta(n^q)$. A major progress was made by Bukh and Conlon \cite{bukh_conlon}, who proved that every rational number between $1$ and $2$ can be realized as the Tur\'{a}n exponent of a finite family of graphs instead of a single graph. Since then, people have derived many partial results concerning a single graph \cite{kang_kim_liu,conlon_janzer_lee,janzer,jiang_jiang_ma,conlon_janzer}, each extending the range of exponents for which the conjecture is known. Despite attracting much attention in recent years, this conjecture on realizing rational exponents remains widely open. 

The first instance of the meta-problem that we propose is the following induced version of the rational exponent conjecture due to Erd\H{o}s and Simonovits. 

\begin{conjecture} \label{conj:ind_rational_exp}
    For any rational $q \in (1, 2)$, there exists a bipartite $H$ with $\ex^*(n, H, s) = \Theta_{s}(n^q)$. 
\end{conjecture}

%If both \Cref{conj:HMST} and \Cref{conj:rational_exp} turn out to be true, then \Cref{conj:ind_rational_exp} trivially holds. 

To realize every rational between $1$ and $2$ as a Tur\'{a}n exponent, Bukh and Conlon \cite{bukh_conlon} constructed certain lifting families of balanced rooted trees. (We shall include the formal definitions in \Cref{sec:rational}.) Their upper bound follows easily from the supersaturation of trees, while their lower bound is an application of the random polynomial construction developed by Bukh \cite{bukh2015}. It is also worth mentioning that such constructions are useful in many other extremal problems \cite{conlon,ma_yuan_zhang,he_tait,bukh_tait,xu_zhang_ge}. 

As our first result, we are going to prove the following towards \Cref{conj:ind_rational_exp}. 

\begin{theorem} \label{thm:ind_rational_exp+}
    For every rational $q = \frac{a}{b} \in (1, 2)$ with $a, b \in \N_+$, there exists a family $\cH$ consisting of at most $2^a$ bipartite graphs such that $\ex^*(n, \cH, s) = \Theta_s(n^q)$. 
\end{theorem}

We remark that the family $\cH$ above is a subfamily of the Bukh--Conlon family \cite[Section 1]{bukh_conlon}. \Cref{thm:ind_rational_exp+} trivially holds if the ``subfamily'' restriction is dropped. Indeed, if we temporarily denote by $\cT$ the Bukh--Conlon bipartite graph family which attains Tur\'{a}n exponent $q \in (1, 2) \cap \mathbb{Q}$, then the larger family $\cT_+$ obtained by including all supergraphs on the same vertex set of some graph in $\cT$ has $\ex^*(n, \cT_+, s) = \Theta_s(n^q)$, achieving induced Tur\'{a}n exponent $q$. 

The bulk of work in proving~\Cref{thm:ind_rational_exp+} is to establish the upper bounds. Our main novelty is to deduce $\ex^*(n, \cH, s) = O_s(n^q)$ for a smaller and more structured subfamily $\cH \subset \cT$ that we find via a Ramsey argument (see~\Cref{rmk:family} for comparison of $\cH$ and $\cT$). Hence our result is slightly stronger than the previous one due to Bukh and Conlon \cite{bukh_conlon}, for we achieve the same (induced) Tur\'{a}n exponent via a smaller family. Unfortunately, this is not strong enough to settle \Cref{conj:ind_rational_exp}; further reducing $\cH$ to a single graph would require new ideas. Also, the short Bukh--Conlon supersaturation argument \cite[Lemma 1.1]{bukh_conlon} has no direct extension to the induced setting, and our proof is more involved. 

Use standard notations $\delta(G)$ and $\Delta(G)$ for the minimum and maximum vertex degree in $G$, respectively. We say that $G$ is \emph{$K$-almost-regular} if $\Delta(G) \le K\delta(G)$, provided that $K > 1$ is some constant. A key step in the proof is the following supersaturation result for induced trees.

\begin{theorem} \label{thm:tree_supersaturation}
    Let $T$ be a $t$-vertex tree. If $G$ is an $n$-vertex $K$-almost-regular $K_{s, s}$-free graph with average degree $d \ge (4Kt)^{6s} s^3$, then there are at least $n(\frac{d}{2K})^{t-1}$ labeled induced copies of $T$ in $G$. 
\end{theorem}

As the host graph is almost regular, our lower bound $\Omega(nd^{t-1})$ is optimal. Later on, we also obtain supersaturation results for even cycles (see \Cref{thm:many_ind_cycle,thm:C4_supersaturation}). These supersaturation results are manifestations of the ``transfering bipartite extremal results to induced setting'' phenomena as well. Considering the importance of supersaturation results and their wide range of applications (see for instance \cite{corsten_tran,pikhurko_yilma,dubroff_gunby_narayanan_spiro}), deducing supersaturation results for various different induced bipartite graphs under local sparsity condition could be an interesting direction for future work.

For lower bound constructions, when $s$ is larger than the order of each graph in $\cT_-$, one can use $\ex(n, \cT) \le \ex(n, \cH) \le \ex^*(n, \cH, s)$. To make the dependency on $s$ better and more explicit, we instead take clique blowups of constructions in \cite[Section 2.3]{bukh_conlon}. This blowup construction will also imply the lower bounds in \Cref{thm:ind_theta,thm:ind_prism} below. 

\begin{proposition} \label{prop:blowup}
    Let $H$ be a connected bipartite graph on $h$ vertices. If $\ex(n, H) = \Omega(n^{1+\alpha})$ for some $\alpha \in (0, 1)$, then for $n \ge s \ge h$ we have $\ex^*(n, H, s) = \Omega_h(s^{1-\alpha}n^{1+\alpha})$. 
\end{proposition}

\subsection{Theta graphs}

\begin{comment}
\centerline{
\begin{tikzpicture}
    \clip (-4.5, -2.75) rectangle (4.5, 2);
    \foreach \x in {-2.5, -1.5, -0.5, 0.5, 1.5, 2.5}
        \foreach \y in {-1.5, -0.75, 0, 0.75, 1.5}
            \draw[fill = black] (\x, \y) circle (0.075);
    \draw[fill = black] (-3.5, 0) circle (0.075);
    \draw[fill = black] (3.5, 0) circle (0.075);
    \foreach \y in {-1.5, -0.75, 0, 0.75, 1.5}
        \draw[thick] (-3.5, 0) -- (-2.5, \y) -- (2.5, \y) -- (3.5, 0);
    \node at (0,-2.25) {\textbf{\hypertarget{figone}{Figure 1:}} The theta graph $\Theta_7^5$. }; 
\end{tikzpicture}
}
\end{comment}

Bounding the even cycle Tur\'{a}n number $\ex(n, C_{2\ell})$ is one of the most well-studied bipartite Tur\'{a}n problems. In general, an upper bound $\ex(n, C_{2\ell}) = O_{\ell} \bigl( n^{1+\frac{1}{\ell}} \bigr)$ was originally claimed by Erd\H{o}s \cite{erdos1964} and first published by Bondy and Simonovits \cite{bondy_simonovits}. Up to now, the aforementioned upper bound is known to be tight---meaning $\ex(n, C_{2\ell}) = \Theta_{\ell} \bigl( n^{1+\frac{1}{\ell}} \bigr)$---only for $\ell = 2, 3, 5$ (see \cite{brown,benson,wenger}). The dependency on $\ell$ has seen many improvements throughout the years \cite{pikhurko,bukh_jiang,he}. 

Although the order of magnitude of $\ex(n, C_{2\ell})$ in general is mysterious, more is known if, instead of forbidding two paths between a pair of vertices, one forbids many such paths. The theta graph $\Theta_{\ell}^t$ consists of $t$ many internally vertex disjoint $\ell$-edge paths sharing exactly the beginning and the ending vertices (so $C_{2\ell} = \Theta_{\ell}^2$). This graph contains $2 + (\ell-1)t$ vertices and $\ell t$ edges. %See \hyperlink{figone}{Figure~1} for an illustration. 
Faudree and Simonovits \cite{faudree_simonovits} deduced that $\ex(n, \Theta_{\ell}^t) = O_{\ell, t} \bigl( n^{1+\frac{1}{\ell}} \bigr)$, and later Conlon \cite{conlon} showed the tightness of this bound provided that $t$ is sufficiently large in $\ell$. Bukh and Tait \cite{bukh_tait} obtained some precise results on the dependency of $\ex(n, \Theta_{\ell}^t)$ in $t$. Despite the correct asymptotics of $\ex(n, C_8)$ is unknown, Verstra\"{e}te and Williford \cite{verstraete_williford} obtained $\ex(n, \Theta_4^3) = \Theta(n^{5/4})$. 

We shall prove the following result, which provides further evidence to the conjecture of Hunter, Milojevi\'{c}, Sudakov, and Tomon~\cite{hunter_milojevic_sudakov_tomon}.

\begin{theorem} \label{thm:ind_theta}
    For any integer $\ell\ge 2$, there exists some sufficiently large $L$, such that 
    \[
    \Omega_{\ell, t} \bigl( s^{1-\frac{1}{\ell}} \cdot n^{1+\frac{1}{\ell}} \bigr) \le \ex^*(n, \Theta_{\ell}^t, s) \le O_{\ell, t} \bigl( (\ell t)^{20s} \cdot n^{1+\frac{1}{\ell}} \bigr)
    \]
    holds for every integer $t \ge L$. 
\end{theorem}

\subsection{Prism graphs}

The prism graph $C^{\square}_{\ell}$ consists of two vertex-disjoint $\ell$-cycles and a perfect matching between the two cycles. See \hyperlink{figtwo}{Figure~1} for an illustration. Notice that $C^{\square}_{\ell}$ is bipartite if and only if $\ell$ is even. 

\centerline{
\begin{tikzpicture}
    \clip (-1, -3) rectangle (6, 2.5);
    
    \foreach \m in {1, 2, 3, 4}
        {
        \foreach \n in {0.75, 1.75, -0.75, -1.75}
            \draw[fill = black] (\m,\n) circle (0.075);
        \draw[thick] (\m,0.75) -- (\m,1.75);
        \draw[thick] (\m,-0.75) -- (\m,-1.75);
        }
    
    \draw[fill = black] (-0.5,0.5) circle (0.075);
    \draw[fill = black] (-0.5,-0.5) circle (0.075);
    \draw[fill = black] (5.5,0.5) circle (0.075);
    \draw[fill = black] (5.5,-0.5) circle (0.075);
    
    \draw[thick] (-0.5,0.5) -- (-0.5,-0.5);
    \draw[thick] (-0.5,0.5) -- (1,1.75);
    \draw[thick] (-0.5,0.5) -- (1,-0.75);
    \draw[thick] (-0.5,-0.5) -- (1,-1.75);
    \draw[thick] (-0.5,-0.5) -- (1,0.75);
    
    \draw[thick] (5.5,0.5) -- (5.5,-0.5);
    \draw[thick] (5.5,0.5) -- (4,1.75);
    \draw[thick] (5.5,0.5) -- (4,-0.75);
    \draw[thick] (5.5,-0.5) -- (4,-1.75);
    \draw[thick] (5.5,-0.5) -- (4,0.75);
    \draw[thick] (1,0.75) -- (4,0.75);
    \draw[thick] (1,1.75) -- (4,1.75);
    \draw[thick] (1,-0.75) -- (4,-0.75);
    \draw[thick] (1,-1.75) -- (4,-1.75);
    
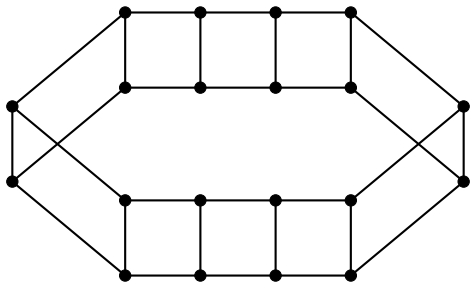
    \node at (2.5,-2.5) {\textbf{\hypertarget{figtwo}{Figure 1:}} The prism graph $C^{\square}_{10}$. }; 
\end{tikzpicture}
}

Observe that $C_4^{\square}$ is the cube graph $Q_8$. It is a notoriously difficult long-standing open problem to determine the order of magnitude of $\ex(n, Q_8)$, and the best known result \cite{erdos_simonovits,pinchasi_sharir} up to now is 
\[
\Omega(n^{3/2}) \le \ex(n, C_4) \le \ex(n, Q_8) \le O(n^{8/5}). 
\]
Partly motivated by the cube problem, people study prism graphs. A celebrated dependent random choice argument \cite{alon_krivelevich_sudakov} shows that $\ex(n, C_{2\ell}^{\square}) = O(n^{5/3})$. He, Li, and Feng \cite{he_li_feng} determined $\ex(n, C_{2\ell+1}^{\square})$ for every $\ell \ge 1$ and large $n$, and conjectured that $\ex(n, C_{2\ell}^{\square}) = o(n^{5/3})$. Very recently, Gao, Janzer, Liu and Xu~\cite{gao_janzer_liu_xu} resolved this conjecture. In particular, they established $\ex(n, C^{\square}_{2\ell}) = \Theta_{\ell} (n^{3/2})$ for every $\ell \ge 4$. 

%Recall that a graph is called $r$-\emph{degenerate} if each of its subgraphs has minimum degree at most $r$. Erd\H{o}s \cite{erdos1981,erdos1983,erdos1988,erdos1993} offered {\$}250 for a proof and {\$}500 for a disproof of the conjecture that for any bipartite graph $H$, we have $\ex(n, H) = O(n^{3/2})$ if and only if $H$ is $2$-degenerate. This was recently disproved by Janzer \cite{janzer_disproof}. Then Gao et al.~\cite{gao_janzer_liu_xu} showed that each $C_{2\ell}^{\square} \, (\ell \ge 4)$ is a counterexample. 

Our last result extends this optimal bound to the induced setting. 

\begin{theorem} \label{thm:ind_prism}
    For any integer $\ell \ge 10$, we have that 
    \[
    \Omega_{\ell}(s^{1/2} \cdot n^{3/2}) \le \ex^*(n, C_{2\ell}^{\square}, s) \le O_{\ell}(6^ss^{\ell^2} \cdot n^{3/2}). 
    \]
\end{theorem}

\medskip 

\noindent\textbf{Organization.}
We begin with collecting some tools and establishing \Cref{prop:blowup} in \Cref{sec:tools}. In~\Cref{sec:rational}, we prove~\Cref{thm:tree_supersaturation} and deduce \Cref{thm:ind_rational_exp+} from it. We then derive the upper bounds of  \Cref{thm:ind_theta,thm:ind_prism} in \Cref{sec:theta,sec:prism}, respectively. Finally, in \Cref{sec:remark} we conclude the paper with some remarks and open problems.

\section{Preliminaries and constructions} \label{sec:tools}

Throughout this paper, we systematically omit non-essential floors and ceilings. 

In a graph $G = (V, E)$, define for any subset of vertices $S \subseteq V$ the \emph{common neighborhood} of $S$ as 
\[
N(S) \eqdef \bigl\{v \in V \setminus S : \text{$uv \in E$ holds for each vertex $u \in S$} \bigr\}. 
\]
Let $\deg(S) \eqdef |N(S)|$ be the \emph{common degree} of $S$ in $G$. When $S = \{u_1, \dots, u_k\}$ is specific, we slightly abuse the notations by writing $N(u_1, \dots, u_k) \eqdef N(S)$ and $\deg(u_1, \dots, u_k) \eqdef \deg(S)$. 

We first introduce a technical result which allows us to pass to a $K$-almost-regular induced subgraph when proving upper bounds in \Cref{thm:ind_rational_exp+,thm:ind_theta,thm:ind_prism}. 

\begin{lemma} \label{lem:almost_reg}
    Suppose $\alpha, C > 0$ and $K \eqdef 2^{\frac{4}{\alpha}+2}$. For $n$-vertex graph $G$ with $e(G) \ge C \cdot n^{1+\alpha}$, there is a $K$-almost-regular induced $m$-vertex subgraph $H$ with $e(H) \ge \frac{C}{4} \cdot m^{1+\alpha}$ and $m \ge \frac{C^{\frac{\alpha+4}{2\alpha+4}}}{K} \cdot n^{\frac{\alpha}{2\alpha+4}}$. 
\end{lemma}

\Cref{lem:almost_reg} is in spirit the same as \cite[Lemma~4.2]{hunter_milojevic_sudakov_tomon}. We remark that one can prove the theorems above only using \cite[Lemma~4.2]{hunter_milojevic_sudakov_tomon}. However, our statement is technically stronger, for the number of vertices in our located subgraph grows with $n$. This advantage will help us to avoid some subtlety in the proofs of \Cref{thm:ind_rational_exp+,thm:ind_theta,thm:ind_prism}, which could be of independent interest in other applications. 

\begin{proof}[Proof of \Cref{lem:almost_reg}]
    We are going to construct a sequence of graphs $G = G_0, G_1, \dots, G_k$, where $G_i$ is an induced subgraph of $G_{i-1}$ (and $G$) with $e(G_i) \ge 2^iC \cdot n_i^{1+\alpha}$ for $n_i \eqdef |V(G_i)|$ and $i = 1, \dots, k$. 

    Write $K' \eqdef 2^{4/\alpha}$ and assume that $G_i \, (i = 0, 1, \dots)$ has been generated. Consider
    \[
    U_i \eqdef \bigl\{ u \in V(G_i) : \deg_{G_i}(u) \ge 2^i K' C \cdot n_i^{\alpha} \bigr\}. 
    \]
    If $\sum_{u \in U_i} \deg_{G_i}(u) < 2^{i-1} C \cdot n_i^{1+\alpha}$, then we terminate the process and set $k \eqdef i$. Otherwise, we can find $A_i \subseteq V(G_i)$ with $|A_i| = \frac{n_i}{2K'}$ such that $\sum_{a \in A_i} \deg_{G_i}(a) \ge 2^{i-1}C \cdot n_i^{1+\alpha}$ (e.g.~we can take $A_i$ to be the set of $\frac{n_i}{2K'}$ highest degree vertices in $G_i$). Independently on $A_i$, we pick $B_i^* \subseteq V(G_i)$ with $|B_i^*| = \frac{n_i}{2K'}$ uniformly at random. For each edge emanated from some vertex in $A_i$, its another vertex appears in $B_i^*$ with probability $\binom{n_i-1}{n_i/(2K')-1} / \binom{n_i}{n_i/(2K')} = \frac{1}{2K'}$. Then
    \[
    \E \bigl( e(G_i[A_i \cup B_i^*]) \bigr) \ge \frac{1}{2} \cdot \frac{1}{2K'} \cdot \sum_{a \in A_i} \deg_{G_i}(a) \ge \frac{2^{i-3}C}{K'} \cdot n_i^{1+\alpha} = 2^{i+1}C \cdot \Bigl(\frac{n_i}{K'}\Bigr)^{1+\alpha}, 
    \]
    where the last equality follows from $K' = 2^{4/\alpha}$. So, we can find $B_i \subseteq V(G_i)$ with $|B_i| = \frac{n_i}{2K'}$ such that $e(G[A_i \cup B_i]) \ge 2^{i+1}C \cdot (\frac{n_i}{K'})^{1+\alpha}$. Generate the induced subgraph $G_{i+1} \eqdef G_i[A_i \cup B_i]$. 

    For each $i$, we have $n_{i+1} \ge \frac{n_i}{2K'}$, and so $n_k \ge \frac{n}{(2K')^k}$. Moreover, 
    \[
    2^k C \cdot n_k^{1+\alpha}\le 2^kC\cdot \Big(\frac{n_{k-1}}{K'}\Big)^{1+\alpha} \le e(G_k) \le n_k^2 \implies 2^k C \le n_k^{1-\alpha} \le n_k \le n. 
    \]
    Let $H$ be the graph obtained from $G_k$ by removing all vertices in $U_k$ and then greedily removing each vertex of degree at most $2^{k-2} C \cdot n_k^{\alpha}$. It then follows from the definition of $U_k$ that the minimum and the maximum degree of $H$ satisfy 
    \[
    2^{k-2} C \cdot n_k^{\alpha} \le \delta(H) \le \Delta(H) \le 2^k K'C \cdot n_k^{\alpha}, 
    \]
    and hence $H$ is $4K'$-almost-regular. The graphs stop at $G_k$ implies $\sum_{u \in U_k} \deg_{G_k}(u) \le 2^{k-1} C \cdot n_k^{1+\alpha}$. Write $m \eqdef |V(H)| \le n_k$. Then
    \[
    e(H) \ge e(G_k) - 2^{k-1} C \cdot n_k^{1+\alpha} - 2^{k-2} C \cdot n_k^{1+\alpha} \ge 2^{k-2} C \cdot n_k^{1+\alpha} \ge \frac{C}{4} \cdot m^{1+\alpha}. 
    \]
    Moreover, from $n_k \ge \frac{n}{(2K')^k}, \, n_k \ge n_k^{1-\alpha} \ge 2^k C$, and $K' = 2^{4/\alpha}$ we deduce that
    \[
    n_k \cdot n_k^{\log_2(2K')} \ge \frac{n}{(2K')^k} \cdot (2^k C)^{\log_2(2K')} = C^{\log_2(2K')} \cdot n \implies n_k\ge C^{\frac{\alpha+4}{2\alpha+4}} \cdot n^{\frac{\alpha}{2\alpha+4}}. 
    \]
    Substituting $K = 2^{4/\alpha+2}$, we conclude that $m \ge \frac{e(H)}{\Delta(H)} \ge \frac{n_k}{4K'} \ge \frac{C^{\frac{\alpha+4}{2\alpha+4}}}{K} \cdot n^{\frac{\alpha}{2\alpha+4}}$. 
\end{proof}

We also recall the K\H{o}v\'{a}ri--S\'{o}s--Tur\'{a}n theorem, a cornerstone result in bipartite Tur\'{a}n problems. 

\begin{theorem}[\cite{kovari_sos_turan}] \label{thm:KST}
    Let $s, t \in \N_+$ with $t \ge s \ge 2$. Then $\ex(n, K_{s,t}) \le \frac{1}{2}(t-1)^{\frac{1}{s}}n^{2-\frac{1}{s}} + \frac{1}{2}(s-1)n$. 
\end{theorem}

\begin{corollary} \label{coro:KST}
    Let $c \in (0, 1)$ and $G$ be an $n$-vertex $K_{s,s}$-free graph. If $n \ge \frac{s-1}{c^s}$, then $e(G) \le cn^2$. 
\end{corollary}

\begin{proof}
    It follows from $n \ge \frac{s-1}{c^s}$ that $s - 1 \le c^sn$. Since $c \in (0, 1)$ hence $c^s \le c$, \Cref{thm:KST} gives
    \[
    e(G) \le \frac{1}{2}(s-1)^{\frac{1}{s}}n^{2-\frac{1}{s}} + \frac{1}{2}(s-1)n \le \frac{1}{2}\bigl( cn^{\frac{1}{s}} \bigr) n^{2-\frac{1}{s}} + \frac{1}{2}(c^sn)n \le cn^2. \qedhere
    \]
\end{proof}

Finally, we include a proof of \Cref{prop:blowup} as promised. 

\begin{proof}[Proof of~\Cref{prop:blowup}]
    Write $t \eqdef \frac{s}{2h}$ and let $G_0$ be an $\frac{n}{t}$-vertex $H$-free graph on $\Omega \bigl( (\frac{n}{t})^{1+\alpha} \bigr)$ edges. Replace each vertex of $G_0$ by a $t$-clique (to be referred to as a \emph{blob}). Between any pair of blobs, 
    \vspace{-0.5em}
    \begin{itemize}
        \item put all $t^2$ edges if there is an edge between the original pair of vertices in $G_0$, and
        \vspace{-0.5em}
        \item put no edge if there is no edge between the original pair of vertices in $G_0$. 
    \end{itemize}
    \vspace{-0.5em}
    Denote by $G$ the resulting
    graph after this blowup process. Then $e(G) \ge t^2 e(G_0) = \Omega_h(s^{1-\alpha}n^{1+\alpha})$, and we are left to verify that $G$ is induced-$H$-free and $K_{s, s}$-free. 
    
    Suppose there is an induced copy of $H$ in $G$. Fix a pair of distinct vertices $u, v$ from this copy. If $u, v$ are from a same part of $H$, then they must live in different blobs of $G$ because each blob is a complete graph. If $u, v$ are from different parts of $H$, then the assumptions that $H$ is connected and $h\ge 3$ due to $\ex(n,H)$ being superlinear show that they have a \emph{distinguisher} $w$ which is adjacent to exactly one of the vertices in $H$. Provided that $u, v$ live in a same blob of $G$, the distinguisher vertex $w$ does not exist in $G$. It follows that $u, v$ always live in different blobs of $G$. So, each vertex of the induced copy $H$ lives in a distinct blob of $G$, and hence $G_0$ contains a copy of $H$, a contradiction. 
    
    Suppose there is a copy of $K_{s, s}$ in $G$. Since each blob is of size $t = \frac{s}{2h}$, we can find a $K_{h, h}$-subgraph of the $K_{s, s}$ whose vertices come from pairwise distinct blobs through a standard greedy process. The existence of such a $K_{h, h}$ in $G$ shows that $G_0$ contains a copy of $K_{h, h}\supseteq H$, a contradiction. 
    
    The proof of \Cref{prop:blowup} is complete. 
\end{proof}

\section{Rational exponents of induced Turán numbers} \label{sec:rational}

A \emph{rooted tree} $(T; R)$ consists of a tree graph $T = (V, E)$ and a root set $R \subset V$ such that $T$ induces an independent set on $R$. For every proper subset $S \subset V \setminus R$ (possibly empty) and any positive integer $p$, we construct a disjoint union of sets $W_S \eqdef \bigcup_{v \in V} W_v$, where $|W_v| = 1$ if $v \in R \cup S$ and $|W_v| = p$ otherwise. Let $W_v = \{v^1, \dots, v^p\}$, where $v^1 = \dots = v^p = v$ if $v \in R \cup S$ and $v^1, \dots, v^p$ are distinct otherwise. The $(p; S)$-\emph{lift} of $(T; R)$ is the graph on vertex set $W_S$, where we put an edge between $u^i$ and $v^j$ if and only if $i = j$ and $uv$ forms an edge in the original tree $T$. In other words, the $(p; S)$-lift of $(T; R)$ is the graph obtained from gluing $p$ vertex-disjoint copies of $T$ along the set $R \cup S$. Denote by $\cF^p(T; R)$ the family (set) of possible $(p; S)$-lifts of $(T; R)$ as $S$ ranges over all proper subsets of $V \setminus R$. Hereafter we implicitly assume that every rooted tree has at least one root vertex and at least one non-root vertex. The \emph{density} of $(T; R)$ is defined as $\rho(T; R) \eqdef \frac{e(T)}{|V\setminus R|}$. 

\begin{center}
\begin{tikzpicture}[scale = 1.25]
    \clip (-3, -1.75) rectangle (3, 1.25);
    \draw[thick] (-2, 0) -- (2, 0);
    \foreach \x in {-2, 2}
        \draw[fill = black] (\x, 0) circle (0.1);
    \foreach \x in {-1, 0, 1}
        \draw[fill = white] (\x, 0) circle (0.1);
    \node at (-2, -0.25) {$a$};
    \node at (-1, -0.25) {$b$};
    \node at (0, -0.25) {$c$};
    \node at (1, -0.25) {$d$};
    \node at (2, -0.25) {$e$};
\end{tikzpicture}
\begin{tikzpicture}[scale = 1.25]
    \clip (-3, -1.75) rectangle (3, 1.25);
    \draw[thick] (-2, 0) -- (2, 0);
    \draw[thick] (-2, 0) -- (-1, 1) -- (0, 1) -- (1, 0);
    \draw[thick] (-2, 0) -- (-1, -1) -- (0, -1) -- (1, 0);
    \foreach \x in {-2, 2}
        \draw[fill = black] (\x, 0) circle (0.1);
    \foreach \x in {-1, 0}
        \foreach \y in {-1, 0, 1}
            \draw[fill = white] (\x, \y) circle (0.1);
    \draw[fill = white] (1, 0) circle (0.1);
    \node at (-2, -0.25) {$a$};
    \foreach \y in {1, 2, 3}
        \node at (-1, -\y+1.75) {$b^{\y}$}; 
    \foreach \y in {1, 2, 3}
        \node at (0, -\y+1.75) {$c^{\y}$}; 
    \node at (1, -0.25) {$d$};
    \node at (2, -0.25) {$e$};
\end{tikzpicture}
\begin{tikzpicture}
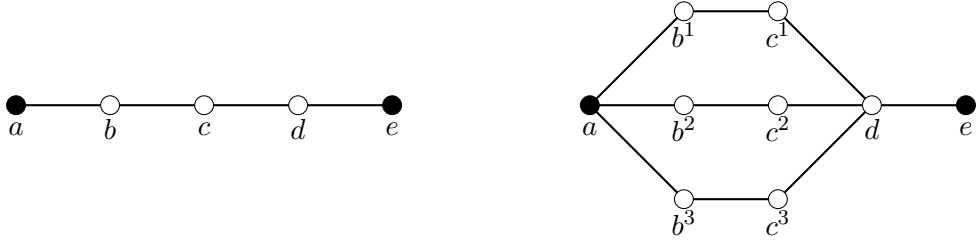
    \node at (0, 0) {\textbf{\hypertarget{figthree}{Figure 2:}} The $(p; S)$-lift of a rooted $4$-edge path where $R = \{a, e\}, \, p = 3$, and $S = \{d\}$.};
\end{tikzpicture}
\end{center}

To conclude \Cref{thm:ind_rational_exp+} using \Cref{prop:blowup}, the following upper bound suffices. 

\begin{theorem} \label{thm:roottree_upper}
    Let $t \ge 2, \, s \ge 2$ be integers and suppose $(T; R)$ is a $t$-vertex rooted tree. There exists $C = C(s, t, p) > 0$ such that every $n$-vertex $K_{s, s}$-free graph $G$ with $e(G) \ge C \cdot n^{2 - \frac{1}{\rho(T; R)}}$ contains an induced $F$-subgraph for some $F \in \cF^p(T; R)$ (i.e., $F$ is a $(p; S)$-lift of $(T; R)$ for some $S$). 
\end{theorem}

%Inspired by \cite[Lemma 1.1]{bukh_conlon}, the key ingredient in the proof of \Cref{thm:roottree_upper} is to establish the following supersaturation result on induced rooted trees. 

%\begin{theorem} \label{thm:tree_supersaturation}
   % Let $T$ be a $t$-vertex tree. If $G$ is an $n$-vertex $K$-almost-regular $K_{s, s}$-free graph with average degree $d \ge (4Kt)^{6s} s^3$, then there are at least $n(\frac{d}{2K})^{t-1}$ labeled induced copies of $T$ in $G$. 
%\end{theorem}

We shall first deduce \Cref{thm:roottree_upper} from \Cref{thm:tree_supersaturation}, and then prove \Cref{thm:tree_supersaturation}. 

\begin{center}
\begin{tikzpicture}[scale = 1.25]
    \clip (-3, -1.25) rectangle (3, 1.25);
    \draw[thick] (-2, 0) -- (2, 0);
    \foreach \x in {-2, 2}
        \draw[fill = black] (\x, 0) circle (0.1);
    \foreach \x in {-1, 0, 1}
        \draw[fill = white] (\x, 0) circle (0.1);
    \node at (-2, -0.25) {$a$};
    \node at (-1, -0.25) {$b$};
    \node at (0, -0.25) {$c$};
    \node at (1, -0.25) {$d$};
    \node at (2, -0.25) {$e$};
\end{tikzpicture}
\begin{tikzpicture}[scale = 1.25]
    \clip (-3, -1.25) rectangle (3, 1.25);
    \draw[thick] (-2, 0) -- (2, 0);
    \draw[thick] (-2, 0) -- (-1, 1) -- (0, 1) -- (1, 0);
    \draw[thick] (-1, 0) arc[start angle=180, end angle=360, radius=1];
    \draw[thick] (0, 0) arc[start angle=180, end angle=360, radius=1];
    \foreach \x in {-2, 2}
        \draw[fill = black] (\x, 0) circle (0.1);
    \foreach \x in {-1, 0}
        \foreach \y in {0, 1}
            \draw[fill = white] (\x, \y) circle (0.1);
    \draw[fill = white] (1, 0) circle (0.1);
    \node at (-1, 0.75) {$p$};
    \node at (0, 0.75) {$q$};
    \node at (-2, -0.25) {$a$};
    \node at (-1.1, -0.25) {$b$};
    \node at (-0.1, -0.25) {$c$};
    \node at (1.1, -0.25) {$d$};
    \node at (2.1, -0.25) {$e$};
\end{tikzpicture}
\begin{tikzpicture}
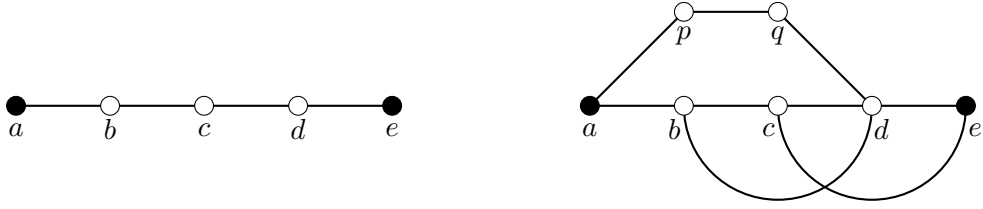
    \node at (0, 0) {\textbf{\hypertarget{figfour}{Figure 3:}} An illustration of a graph living in $\cG^3 \setminus \cF^3$.};
\end{tikzpicture}
\end{center}

\begin{remark}\label{rmk:family}
It is worth mentioning that our $\cF^p(T; R)$ is smaller than the Bukh--Conlon family $\cG^p(T; R)$ (see \cite[Definition 1.2]{bukh_conlon}, denoted as $\cT_R^p$), which they refer to as the $p$-th power of a fixed rooted tree $(T; R)$. The family $\cG^p(T; R)$ consists of all possible unions of $p$ distinct labeled copies of $T$ which agree on $R$. Omit $(T; R)$ when there is no confusion. For the rooted $5$-path in \hyperlink{figthree}{Figure~2}, the graph in the right of \hyperlink{figfour}{Figure~3} lives in $\cG^3$ (since $apqde, \, abcde, \, abdce$ in these orders are three distinct labeled copies of the $5$-path agreeing on $R$) but not in $\cF^3$. For every rooted tree $(T; R)$, the definition implies $\cF^p \subseteq \cG^p$. In fact, the latter is significantly larger than the former, in the sense that $|\cG^p|$ grows with $p$ but $|\cF^p|$ does not. Indeed, suppose the number of non-root vertices in $(T; R)$ is $a$. Since the lifting set $S$ is a subset of $V(T) \setminus R$, we have $|\cF^p| \le 2^a$. For any $\frac{a}{b} \in (0, 1)$ with $a, b \in \N_+$, we can use the rooted trees $T_{a, b}$ in \cite[Section 1]{bukh_conlon} of density $\frac{b}{a}$ and $a$ non-root vertices to get the bound on $|\cH|$ in~\Cref{thm:ind_rational_exp+}. 
%On the other hand, by embedding the non-root vertices of each $T$ into consecutive positive integers (as vertices), counting the number of vertices of the resulting graph shows that $|\cG^p| \ge px - p - x$. This lower bound is generally quite inaccurate, yet essentially tight when $T$ is a star and $R$ is its center.
\end{remark}

\subsection{Proof of \texorpdfstring{\Cref{thm:roottree_upper}}{Theorem 3.1} assuming \texorpdfstring{\Cref{thm:tree_supersaturation}}{Theorem 3.2}}

Denote by $x, y$ the number of non-root vertices and edges in $(T; R)$, respectively. Write $\alpha \eqdef 1 - \frac{x}{y}$. Then $2 - \frac{1}{\rho(T; R)} = 1 + \alpha$, and hence $e(G) \ge C \cdot n^{1+\alpha}$ for some sufficiently large $C$ in terms of $s$ and $t, p$. Thanks to \Cref{lem:almost_reg}, we may assume that $G$ is $K$-almost-regular, where $K \eqdef 2^{\frac{4}{\alpha}+2}$. Notice that the estimate ``$m \ge \frac{C^{\frac{\alpha+4}{2\alpha+4}}}{K} \cdot n^{\frac{\alpha}{2\alpha+4}}$'' in \Cref{lem:almost_reg} shows that we can still assume that $n \to \infty$. In particular, we do not need to worry about the restriction ``$d \ge (4Kt)^{6s} s^3$'' in \Cref{thm:tree_supersaturation}. 

In $T$, the number of non-root vertices $x$ and edges $y$ satisfy $t = y + 1, \, \alpha = 1 - \frac{x}{y}$. Due to the pigeonhole principle, \Cref{thm:tree_supersaturation} shows that there exists $U \subset V(G)$ with $|U| = |R|$ emanating
\[
\frac{n(\frac{d}{2K})^{t-1}}{n(n-1) \cdots (n-t+x+1)} \ge \Bigl( \frac{C}{4K} \Bigr)^{t-1} \cdot \frac{n^{1+\alpha(t-1)}}{n^{t-x}} = \Bigl( \frac{C}{4K} \Bigr)^{t-1} \cdot \frac{n^{1+(1-\frac{x}{y})y}}{n^{y-x+1}} \ge \frac{C}{4K}
\]
induced copies of $(T; R)$ under a fixed orientation on $U$ (i.e., the mapping from $R$ to $U$ is also fixed). 

Up to now, we have embedded $r \eqdef \frac{C}{4K}$ copies of $(T; R)$, all of whose roots on $U$ share a certain orientation. Every copy is induced, yet the interplay between two copies can be messy---any possible intersections are allowed, and the union of $r$ copies lives in $\cT^{r}$. We clean this up in two phases: 
\vspace{-0.5em}
\begin{itemize}
    \item Firstly, using Ramsey argument, among $r$ induced copies we are going to find $q$ of them whose intersections are ``regular''---each vertex of $T$ is mapped to either a same vertex or $q$ distinct vertices, and the union of $q$ copies lives in $\cF^q(T; R)$. 
    \vspace{-0.5em}
    \item Secondly, due to the $K_{s, s}$-free property, in $q$ induced copies with regular intersection we can find $p$ among which no redundant edges go between any pair of them. These $p$ copies of them form an induced copy of a graph in $\cF^p(T; R)$, finishing the proof. 
\end{itemize}

The following lemma carryies out the Ramsey argument in the first phrase.

\begin{lemma} \label{lem:selection}
    Let $V$ be a finite set. Let $t, q \in \N_+$ and write $N = N(t, q) \eqdef (t!)^2 \cdot q^{t+1}$. Among any vectors $\vx_1, \dots, \vx_N \in V^t$, each of distinct entries, there are $\vy_1, \dots, \vy_q$ such that the followings hold.
    \vspace{-0.5em}
    \begin{itemize}
        \item For $j = 1, \dots, t$, the entries $(\vy_1)_j, \dots, (\vy_q)_j$ are either all the same or pairwise distinct. 
        \vspace{-0.5em}
        \item The sets $Y_1, \dots, Y_t$ are pairwise disjoint, where $Y_j \eqdef \bigl\{(\vy_1)_j, \dots, (\vy_q)_j\bigr\}$ for $j = 1, \dots, t$. 
    \end{itemize}
    \vspace{-0.5em}
\end{lemma}

\begin{proof}
    For an arbitrarily fixed integer $q$, we prove by induction on $t$. Denote by
    \[
    M_t \eqdef (t!)^2 \cdot q^{t+1} = N(t, q)
    \]
    for any positive integer $t$. Notice that we always think about $q$ as a constant. 
    
    When $t = 1$, each vector $\vx$ has a single entry. Among $q^2$ such vectors, either $q$ of them have the same entry or $q$ of them are pairwise distinct. These $q$ vectors (as $\vy_1, \dots, \vy_q$) fulfill the properties. 

    We assume that the statement has already been established for $t-1 \, (t \ge 2)$. Think about the vectors $\vx_1, \dots, \vx_N \in V^t$ as being indexed by $[t]$. We then refer to each $\alpha \in [t]$ as a \emph{position}. 
    \vspace{-0.5em}
    \begin{itemize}
        \item If there exists a position $\alpha_0 \in [t]$ such that some $v_0 \in V$ appears at least $M_{t-1}$ times among $(\vx_1)_{\alpha_0}, \dots, (\vx_N)_{\alpha_0}$, then we remove all the vectors whose $\alpha_0$-th entry is not $v_0$. Since each $\vx$ has distinct entries, no $v_0$ appears once we forget about the $\alpha_0$ position of each (remaining) vector. The inductive hypothesis on $t-1$ then implies the existence of $\vy_1, \dots, \vy_q$, as desired. 
        \vspace{-0.75em}
        \item Otherwise, any entry appears no more than $M_{t-1} - 1$ times in each position. Fix $\vy_1 \eqdef \vx_1$ and delete all vectors sharing at least one entry (regardless of its position) with $\vx_1$. Other than $\vx_1$, we are left with at least
        \[
        N - t^2(M_{t-1} - 1) - 1 \ge N - t^2M_{t-1}
        \] 
        many vectors. Fix an arbitrary remaining vector as $\vy_2$ and delete all vectors sharing at least one entry with it. Again we are deleting at most $t^2(M_{t-1} - 1)$ vectors, and hence left with at least $N - 2t^2M_{t-1}$ many vectors other than $\vy_1, \vy_2$. Since $N = M_t = qt^2 \cdot M_{t-1}$, we can repeat this ``fixing and deleting'' process to obtain $\vy_1, \dots, \vy_q$ whose entries are pairwise disjoint. 
    \end{itemize}
    \vspace{-0.5em}
    The casework above finishes the inductive step, and so the proof is complete. 
\end{proof}

When $r \ge (t!)^2 \cdot q^{t+1}$, according to \Cref{lem:selection}, we can find $q$ induced copies $T_1, \dots, T_q$ of $(T; R)$ on root set $U \subset V(G)$ under a fixed orientation in $G$. Moreover, if we denote by $\varphi_i$ the embeddings from $T$ to $T_i$ for $i = 1, \dots, q$, then the non-root vertices $w$ of $T$ have the following properties: 
\vspace{-0.5em}
\begin{itemize}
    \item either $\varphi_1(w) = \dots = \varphi_q(w)$ or $\varphi_1(w), \dots, \varphi_q(w)$ are pairwise distinct; and
    \vspace{-0.5em}
    \item for $w \in V(T) \setminus R$, the sets $\bigl\{ \varphi_1(w), \dots, \varphi_q(w) \bigr\}$ are pairwise disjoint. 
\end{itemize}
\vspace{-0.5em}
Write $W \eqdef \bigcup\limits_{i=1}^q V(T_i)$ with $|W| \le qt$. Then from the K\H{o}v\'{a}ri--S\'{o}s--Tur\'{a}n theorem (\Cref{thm:KST}) we deduce that the $K_{s, s}$-free induced subgraph $G[W]$ satisfies $e(G[W]) \le s^{\frac{1}{s}} (qt)^{2-\frac{1}{s}}$. Denote
\[
S \eqdef \bigl\{ w \in V(T) : \varphi_1(w) = \dots = \varphi_q(w) \bigr\}. 
\]

Consider an auxiliary graph $H$ on vertex set $[q]$, where we put an edge between $i, j$ if and only if the union of $T_i, T_j$ is not an induced subgraph of $G$ (i.e., there is some other edge in between). We are supposed to prove $\alpha(H) \ge p$, for an independent set of size $p$ corresponds to an induced copy of a $(p; S)$-lift of $(T; R)$. If $\alpha(H) < p$, then Tur\'{a}n's theorem implies that
\[
s^{\frac{1}{s}} (qt)^{2-\frac{1}{s}} \ge e(G[W]) \ge e(H) > \frac{q^2}{2p} \implies q < (2p)^s s t^{2s-1}. 
\]
So, we have $\alpha(H) \ge p$ provided that $q \eqdef (2p)^s s t^{2s-1}$. Therefore, upon choosing $C = C(s, t, p)$ such that $\frac{C}{4K} = r \eqdef (t!)^2 \cdot q^{t+1}$ (required in \Cref{lem:selection}) we conclude the proof of \Cref{thm:roottree_upper}. 

\subsection{Proof of \texorpdfstring{\Cref{thm:tree_supersaturation}}{Theorem 3.2}}

We shall embed an induced copy of $T$ into $G$ greedily. The point is to expose $T$ vertex by vertex as $v_1, \dots, v_t$ so that the latest embedded vertex $v_k$ is always a leaf in the induced subtree $T[v_1, \dots, v_k]$. Since $G$ is $K$-almost-regular, we are able to find lots of candidates for each $v_k$. To guarantee an induced copy, the key is to avoid embedding $v_k$ into a \emph{bad neighbor} of $v_{k'}$, the unique neighbor of $v_k$ in $T[v_1, \dots, v_k]$, whose codegree with some of the previously embedded vertices $v_1, \dots, v_{k-1}$ is large. Thanks to the $K_{s, s}$-free property of $G$, the number of such bad neighbors is small. 

With hindsight, we choose parameter $\beta \eqdef 1 - \frac{1}{3s}$. For each vertex $v \in V$, set
\[
X(v) \eqdef \bigl\{u \in V \setminus \{v\} : \deg(u, v) \ge d^{\beta}\bigr\}. 
\]
Recall that the average degree of $G$ is $d = \frac{2e(G)}{n} \ge (4Kt)^{6s} s^3$. 

\begin{claim} \label{claim:X_set}
    For every vertex $v \in V$, we have $|X(v)| \le d^{\beta}$. 
\end{claim}

\begin{poc}
    Write $A \eqdef N(v)$ and $B \eqdef X(v)$. Suppose, to the contrary, that $|B| > d^{\beta}$ and fix a $d^{\beta}$-element subset $B' \subseteq B$. We consider the induced subgraph $G' \eqdef G[A \cup B']$. Notice that $A$ and $B'$ are not necessarily disjoint, and the definition of $X$ implies that $e(G') \ge |B'| \cdot \frac{d^{\beta}}{2} = \frac{d^{2\beta}}{2}$. 

    It follows from $G$ is $K$-almost-regular together with the fact $\delta(G) \le d$ that $|A| \le \Delta(G)\le Kd$. So, by applying the K\H{o}v\'{a}ri--S\'{o}s--Tur\'{a}n theorem (\Cref{thm:KST}) to $G'$, we obtain
    \[
    \frac{d^{2\beta}}{2} \le e(G') \le s^{\frac{1}{s}}(|A|+|B'|\bigr)^{2 - \frac{1}{s}} < s^{\frac{1}{s}} (2Kd)^{2-\frac{1}{s}} \implies d < 2^{3s} (2K)^{6s-3} s^3, 
    \]
    which contradicts the assumption $d \ge (4Kt)^{6s} s^3$. We conclude that $|X(v)| = |B| \le d^{\beta}$. 
\end{poc}

    Label the vertices of $T$ as $w_1, \dots, w_t$ such that the induced graph $T[w_1, \dots, w_k]$ is a tree in which $w_k$ is a leaf ($k = 1, \dots, t$). Such a labeling can be found by the breadth-first-search algorithm. 

    Execute the embedding algorithm as follows. Embed $w_1$ into an arbitrary vertex $v_1 \in V$. Given that $w_1, \dots, w_{k-1} \, (2 \le k \le t)$ have already been embedded into $v_1,v_2,\ldots,v_{k-1}$, respectively and $w_{k'}$ is the unique neighbor of $w_k$ in $T[w_1, \dots, w_{k-1}]$, we embed $w_k$ into an arbitrary vertex 
    \[
    v_k \in V_k \eqdef N(v_{k'}) \setminus \Biggl( \biggl( \bigcup_{i \in [k-1]} X(v_i) \biggr) \cup \biggl( \bigcup_{i \in [k-1] \setminus \{k'\}} N(v_i) \biggr) \Biggr). 
    \]
    Then $v_{\ell}v_k \, (\ell < k)$ is an edge if and only if $\ell = k'$. Hence, the algorithm outputs an induced $T$. 
    
    For $k = 2, \dots, t$, set $\overline{V}_k \eqdef N(v_{k'}) \setminus \Bigl( \bigcup\limits_{i \in [k-1]} X(v_i) \Bigr)$. Since $\delta(G) \ge \frac{\Delta(G)}{K} \ge \frac{d}{K}$, \Cref{claim:X_set} implies
    \[
    |\overline{V}_k| \ge |N(v_{k'})| - \sum_{i=1}^{k-1} |X(v_i)| \ge \frac{d}{K} - td^{\beta}. 
    \]
    We refer to a vertex $v \in \overline{V}_k$ as $\ell$-\emph{bad} if $v \in N(v_{\ell})$, where $\ell \in [k-1] \setminus \{k'\}$. If $v$ is $\ell$-bad, then both $vv_{k'}$ and $vv_{\ell}$ are edges of $G$. Write $\{a, b\} \eqdef \{k', \ell\}$ so that $a < b$. This implies that $v \in N(v_a, v_b)$ and $v_b \notin X(v_a)$. Crucially, we infer from this and the definition of $X$ and \Cref{claim:X_set} that the number of such $\ell$-bad vertices is upper bounded by $d^{\beta}$. Since $d \ge (4Kt)^{6s}$, the candidate set $V_k$ for $v_k$ satisfies
    \[
    |V_k| = |\overline{V}_k| - \Biggl| \bigcup_{\ell \in [k-1] \setminus \{k'\}} \{v \in \overline{V}_k : \text{$v$ is $\ell$-bad}\} \Biggr| \ge \Bigl(\frac{d}{K} - td^{\beta}\Bigr) - td^{\beta} \ge \frac{d}{2K}. 
    \]
    To sum up, there are $n$ choices for $v_1$ and $\frac{d}{2K}$ choices for each of $v_2, \dots, v_t$ in this order. Thus, the number of induced copies of $T$ in $G$ is lower bounded by $n(\frac{d}{2K})^{t-1}$, concluding the proof of~\Cref{thm:tree_supersaturation}. 

\section{Induced Turán numbers of theta graphs} \label{sec:theta}

This section is devoted to the proof of  \Cref{thm:ind_theta}. \Cref{prop:blowup} together with \cite[Theorem 1]{conlon} gives the lower bound in \Cref{thm:ind_theta}. The following implies the upper bound.

\begin{theorem} \label{thm:theta_upper}
     Let $\ell \ge 2, \, t \ge 2$, and $s \ge 2$ be integers. There exists $C = C(\ell, t) > 0$ such that every $n$-vertex $K_{s, s}$-free graph $G$ with $e(G) \ge C \cdot (\ell t)^{20s} \cdot n^{1+\frac{1}{\ell}}$ contains an induced $\Theta_\ell^{t}$-subgraph. 
\end{theorem}

For two graphs $H$ and $G$, denote by $\Hom(H, G)$ the set of homomorphisms from $H$ to $G$. We will systematically abuse a graph homomorphism and its homomorphic image, and so $\Hom(C_{2\ell}, G)$ consists of all $2\ell$-edge closed walks (or $2\ell$-closed-walks) in $G$. Write $\hom(H, G) \eqdef \lvert\Hom(H, G)\rvert$. 

Inspired by \cite[Theorem 1.4]{hunter_milojevic_sudakov_tomon}, the main step towards \Cref{thm:theta_upper} is a strengthening of known arguments showing that most closed $2\ell$-edge walks (or $2\ell$-walks) in $G$ are induced $2\ell$-cycles. 

\begin{theorem} \label{thm:many_ind_cycle}
    Let $t \ge 2$ and $G$ be an $n$-vertex $K$-almost-regular $K_{s, s}$-free graph with $K = 2^{4\ell+2}$. There exists some constant $C = C(\ell, t) > 0$ such that the following holds. If $e(G) > C \cdot (\ell t)^{20s} \cdot n^{1+\frac{1}{\ell}}$, then the proportion of induced $2\ell$-cycles in $\Hom(C_{2\ell}, G)$ is at least $1 - \frac{1}{2t^2}$. 
\end{theorem}

Since even cycles are Sidorenko, \Cref{thm:many_ind_cycle} presents a supersaturation result of induced even cycles. We shall first deduce \Cref{thm:theta_upper} from \Cref{thm:many_ind_cycle}, and then prove \Cref{thm:many_ind_cycle}. 

\subsection{Proof of \texorpdfstring{\Cref{thm:theta_upper}}{Theorem 4.1} assuming \texorpdfstring{\Cref{thm:many_ind_cycle}}{Theorem 4.2}} \label{sec:theta_upper}

Thanks to \Cref{lem:almost_reg}, we may assume that $G$ is $K$-almost-regular, where $K \eqdef 2^{4\ell+2}$. Suppose $C$ is sufficiently large in $\ell, t$ and $n$ goes to infinity. (Specifically, $C > (\ell t)^{100\ell}$.) Write $f(s) \eqdef (\ell t)^{20s}$. Then $G$ is $K$-almost-regular and $K_{s, s}$-free with $e(G) \ge C f(s) \cdot n^{1+\frac{1}{\ell}}$ and $d \eqdef 2 e(G)/n \ge 2C f(s) \cdot n^{1/\ell}$. 

Observe that every homomorphic $2\ell$-cycle $W$ is a $2\ell$-closed-walk and vice versa. Denote
\begin{align*}
    C_{2\ell} &\eqdef \bigl(\{v_0, v_1, \dots, v_{2\ell-1}\}, \{v_0v_1, v_1v_2, \dots, v_{2\ell-1}v_0\}\bigr) \\
    \cZ &\eqdef \bigl\{ \varphi \in \Hom(C_{2\ell}, G) : \text{$\varphi(v_0), \varphi(v_1), \dots, \varphi(v_{2\ell-1})$ do not form an induced cycle in $G$} \bigr\}. 
\end{align*}
Since $\varphi(v_0) \varphi(v_1) \cdots \varphi(v_{2\ell-1})$ is a $2\ell$-closed-walk, \Cref{thm:many_ind_cycle} implies that $|\cZ| \le \frac{\hom(C_{2\ell}, G)}{2t^2}$. 

For any pair of (not necessarily distinct) vertices $u, v$ in $G$, we denote
\begin{align*}
    \cW_{u, v} &\eqdef \bigl\{ \text{$\ell$-walks between vertices $u$ and $v$} \bigr\}, \\
    \cX_{u, v} &\eqdef \bigl\{ (P_1, \dots, P_t) \in \cW_{u, v}^{t} : \text{the union of $P_i, P_j$ is not an induced $C_{2\ell}$ in $G$ for some $i \ne j$} \bigr\}, \\
    \cZ_{u, v} &\eqdef \bigl\{ \varphi \in \cZ : \varphi(v_0) = u, \, \varphi(v_\ell) = v \bigr\}. 
\end{align*}
Then
$|\cX_{u, v}| \le t^2|\cZ_{u,v}| |\cW_{u, v}|^{t-2}$. Upon setting $\Omega \eqdef \bigl\{ (u, v) \in V(G)^2 : \cW_{u, v} \ne \varnothing \bigr\}$ we have
\[
\sum_{(u, v) \in \Omega} \frac{|\cX_{u, v}|}{|\cW_{u, v}|^{t-2}} \le t^2 \sum_{(u, v) \in \Omega} |\cZ_{u, v}| = t^2|\cZ| < \hom(C_{2\ell}, G). 
\]
It follows that
\[
\sum_{(u, v) \in \Omega} |\cW_{u, v}|^2 = \sum_{(u, v) \in V(G)^2} |\cW_{u, v}|^2 =  \hom(C_{2\ell}, G) > \sum_{(u, v) \in \Omega} \frac{|\mathcal{X}_{u,v}|}{|\mathcal{W}_{u,v}|^{t-2}}. 
\]
So, there exists a pair of vertices $u_0, v_0$ with $|\cW_{u_0, v_0}|^2 > \frac{|\cX_{u_0, v_0}|}{|\cW_{u_0, v_0}|^{t-2}}$. This implies that $|\cW_{u_0, v_0}| > 0$ and $|\cX_{u_0, v_0}| < |\cW_{u_0, v_0}|^t$ (hence $u_0 \ne v_0$). Therefore, we can find $t$ paths $P_1, \dots, P_t$ between $u_0, v_0$, each pair of which forms an induced $2\ell$-cycle in $G$. We thus conclude that the union of $P_1, \dots, P_t$ is an induced $\Theta_{\ell}^t$-subgraph of $G$, finishing the proof of \Cref{thm:theta_upper}. 

\subsection{Proof of \texorpdfstring{\Cref{thm:many_ind_cycle}}{Theorem 4.2}} \label{sec:many_ind_cycle}

To be specific, we assume that $C > (\ell t)^{100\ell}$. For any homomorphic $2\ell$-cycle $W$ (or a $2\ell$-closed-walk) in $G$, call it \emph{degenerate} if $|V(W)| < 2\ell$. We begin with introducing the following tools. 

\begin{lemma}[{\cite[Lemma 4.3]{hunter_milojevic_sudakov_tomon}}] \label{lem:degenarate-C_2k}
    Let $G$ be an $n$-vertex $K$-almost-regular graph with $K = 2^{4\ell+2}$ and $f(s) = (\ell t)^{20s}$. If $e(G) \ge C f(s) \cdot n^{1 + \frac{1}{\ell}}$, then at most $\frac{2^{3\ell+10}}{\sqrt{Cf(s)}} \hom(C_{2\ell}, G)$ homomorphic $2\ell$-cycles in $G$ are degenerate. 
\end{lemma}

\begin{lemma}[{\cite[Lemma 4.6]{hunter_milojevic_sudakov_tomon}}] \label{lem:red-blue}
    For every $k \in \N_+$ and every $\lambda \in (0, 1)$, if $\delta \in \bigl( 0, \frac{\lambda^{3/2}}{100k} \bigr)$ and integer $m > \lambda\delta^{-1}$, then the following holds. Let $G_R, G_B$ be a red and a blue graph on a common $m$-vertex set $V$. If $e(G_R) \le \delta m^2$ and $e(G_B) \ge \lambda m^2$, then there exists $S \subseteq V$ with $|S| \ge k$ such that the induced red and blue subgraphs satisfy that $e(G_R[S]) = 0$ and $e(G_B[S]) \ge \frac{\lambda}{2} |S|^2$. 
\end{lemma}

The constants in \Cref{lem:degenarate-C_2k} are slightly adjusted from the original version. For \Cref{lem:red-blue}, the specific range of $\delta$ comes from the proof of \cite[Lemma 4.6]{hunter_milojevic_sudakov_tomon}. 

\medskip

Set $\veps \eqdef \frac{1}{2t^2}$. Suppose to the contrary that at least $\veps$-fraction of homomorphic $2\ell$-cycles in $G$ are not induced $2\ell$-cycles. Write $\Gamma \eqdef \bigl\{ (u, v) \in V(G)^2 : u \ne v \bigr\}$. For any $(u, v) \in \Gamma$, we denote
\begin{align*}
    \cP_{u, v} &\eqdef \bigl\{ \text{$\ell$-edge paths from $u$ to $v$} \bigr\}, \\
    \cA_{u, v} &\eqdef \bigl\{ (P_1, P_2) \in \cP_{u, v}^2 : \text{there is some vertex shared by $V(P_1) \setminus \{u, v\}$ and $V(P_2) \setminus \{u, v\}$} \bigr\}, \\
    \cB_{u, v} &\eqdef \bigl\{ (P_1, P_2) \in \cP_{u, v}^2 : \text{there is some edge between $V(P_1) \setminus \{u, v\}$ and $V(P_2) \setminus \{u, v\}$} \bigr\}. 
\end{align*}
The indices ``$u, v$'' are ordered, and every path is a homomorphic path (hence the orientation matters). Note also that we can identify $(P_1, P_2) \in \cP_{u, v}^2$ with the homomorphic $2\ell$-cycle $u \, \overset{P_1}{\cdots} \, v \, \overset{P_2}{\cdots}$. Thus,
\[
\{ \text{non-degenerate homomorphic $2\ell$-cycles} \} \subseteq \bigcup_{(u, v) \in \Gamma} \cP_{u, v}^2 \subseteq \{ \text{homomorphic $2\ell$-cycles} \}. 
\]
Since $C$ is large in terms of $\ell, t$, \Cref{lem:degenarate-C_2k} shows that the proportion of degenerate $2\ell$-cycles among homomorphic $2\ell$-cycles in $G$ is upper bounded by $\frac{2^{3\ell+10}}{\sqrt{C f(s)}} \le \frac{\veps}{2}$. This implies that
\begin{equation} \label{eq:Puv_bound}
\hom(C_{2\ell}, G) \ge \sum_{(u, v) \in \Gamma} |\cP_{u, v}|^2 \ge \Bigl( 1 - \frac{\veps}{2} \Bigr) \hom(C_{2\ell}, G) \ge \frac{1}{2} \hom(C_{2\ell}, G). 
\end{equation}
For each $(u, v)\in\Gamma$, every element of $\cA_{u, v}$ corresponds to a degenerate homomorphic $2\ell$-cycle in $G$. Furthermore, every degenerate homomorphic $2\ell$-cycle is counted in at most one such set $\cA_{u, v}$. (A degenerate $2\ell$-cycle whose $1$-st and $(\ell+1)$-th vertices coincide is never counted.) Thus, 
\begin{equation} \label{eq:Auv_bound}
    \sum_{(u, v) \in \Gamma} |\cA_{u, v}| \le |\{\text{degenerate homomorphic $2\ell$-cycles}\}| \le \frac{2^{3\ell+10}}{\sqrt{Cf(s)}} \hom(C_{2\ell}, G). 
\end{equation}
Recall that at least $\veps$-fraction of homomorphic $2\ell$-cycles are non-induced. This then implies that at least $\frac{\veps}{2}$-fraction of homomorphic $2\ell$-cycles are neither degenerate nor induced. Such a cycle always has a chord, and at least one of its $2\ell$ cyclic shifts lives in the union of $\cB_{u, v}$. Therefore, 
\begin{equation} \label{eq:Buv_bound}
    \sum_{(u, v) \in \Gamma} |\mathcal{B}_{u, v}| \ge \frac{|\{ \text{non-degenerate non-induced homomorphic $2\ell$-cycles} \}|}{2\ell} \ge \frac{\varepsilon}{4\ell} \hom(C_{2\ell}, G). 
\end{equation} 

Sidorenko's property \cite{sidorenko} of $C_{2\ell}$ implies $\hom(C_{2\ell}, G) \ge d^{2\ell}$. Due to \eqref{eq:Puv_bound}, \eqref{eq:Auv_bound}, and \eqref{eq:Buv_bound}, we have that
\[
\frac{8\ell}{\veps} \sum_{(u, v) \in \Gamma}|\cB_{u, v}|-\frac{\sqrt{Cf(s)}}{2^{3\ell+10}} \sum_{(u, v) \in \Gamma} |\cA_{u, v}| \ge \hom(C_{2\ell}, G) \ge \frac{1}{2} d^{2\ell} + \frac{1}{2} \sum_{(u, v) \in \Gamma} |\cP_{u, v}|^2. 
\]
So, there exists a pair of distinct vertices $u_*, v_*$ such that
\[
|\cB_{u_*, v_*}| \ge \frac{\veps\sqrt{C f(s)}}{\ell \cdot 2^{3\ell+13}} \cdot | \cA_{u_*, v_*}| + \frac{\veps}{16\ell} \cdot \frac{d^{2\ell}}{n(n-1)} + \frac{\varepsilon}{16\ell} \cdot |\cP_{u_*, v_*}|^2. 
\]
The definition of $\cB$ implies $\cB_{u_*, v_*} \subseteq \cP_{u_*, v_*}^2$ hence $|\cB_{u_*, v_*}| \le |\cP_{u_*, v_*}|^2$. We thus obtain
\[
|\cA_{u_*, v_*}| \le \frac{\ell \cdot 2^{3\ell+13}}{\veps \sqrt{C f(s)}} \cdot |\cP_{u_*, v_*}|^2, \qquad |\cB_{u_*, v_*}| \ge \frac{\veps}{16\ell} \cdot |\cP_{u_*, v_*}|^2, \qquad |\cP_{u_*, v_*}| \ge \sqrt{\frac{\varepsilon}{16\ell}} \cdot \bigl( Cf(s) \bigr)^{\ell}. 
\]

Construct a red auxiliary graph $G_R$ and a blue auxiliary graph $G_B$ on the vertex set $\cP_{u_*, v_*}$. 
\vspace{-0.5em}
\begin{itemize}
    \item Include $P_1P_2$ as a red edge into $E(G_R)$ if $(P_1, P_2) \in \cA_{u_*, v_*}$. 
    \vspace{-0.5em}
    \item Include $P_1P_2$ as a blue edge into $E(G_B)$ if $(P_1, P_2) \in \cB_{u_*, v_*}$. 
\end{itemize}
\vspace{-0.5em}
Since $(P_1, P_2) \in \cA_{u_*, v_*}$ (resp.~$\cB_{u_*, v_*}$) if and only if $(P_2, P_1) \in \cA_{u_*, v_*}$ (resp.~$\cB_{u_*, v_*}$), the above edge sets are well-defined. Counting the total number of red and blue edges, we obtain
\[
e(G_R) \le |\cA_{u_*, v_*}| \le \frac{\ell \cdot 2^{3\ell+13}}{\veps \sqrt{C f(s)}} \cdot |\cP_{u_*, v_*}|^2, \qquad e(G_B) \ge \frac{|\cB_{u_*, v_*}|}{2} \ge \frac{\veps}{32\ell} \cdot |\cP_{u_*, v_*}|^2. 
\]

We are going to apply \Cref{lem:red-blue} to $G_R$ and $G_B$ with parameters 
\[
(k, \lambda, \delta, m) = \biggl( s \cdot (8 \ell t)^{3s}, \frac{\veps}{32\ell}, \frac{\ell \cdot 2^{3\ell+13}}{\veps\sqrt{C f(s)}}, |\cP_{u_*, v_*}| \biggr). 
\]
Before applying the lemma, we need to verify that both $\delta < \frac{\lambda^{3/2}}{100k}$ and $m > \lambda\delta^{-1}$ hold. These are true because $C$ is sufficiently large in $\ell, t$ and $f(s) = (\ell t)^{20s}$. Therefore, by \Cref{lem:red-blue} we can find internally disjoint paths $P_1, \dots, P_r$ with $r \ge k$ such that at least $\frac{\delta r^2}{2} = \frac{\veps r^2}{64\ell}$ pairs among them appear in $\cB_{u_*, v_*}$. Write $U \eqdef \bigcup_{i=1}^r V(P_i)$ and consider the induced subgraph $H \eqdef G[U]$. Observe that
\[
|U| = 2 + r(\ell-1) \in \bigl( r(\ell-1), r\ell \bigr), \qquad e(H) \ge \frac{1}{2} \cdot \frac{\veps r^2}{64\ell} = \frac{\veps r^2}{128\ell}. 
\]
Choosing parameter $c \eqdef \frac{\veps}{128\ell^3}$, it follows from $\veps = \frac{1}{2t^2}$ that 
\[
|U| > k(\ell-1) > (s-1) \cdot (256 \ell^3 t^2)^s = \frac{s-1}{c^s}. 
\]
Since $G$ (hence $H$) is $K_{s, s}$-free, \Cref{coro:KST} applied to $H$ implies that
\[
\frac{\veps r^2}{128\ell} \le e(H) \le c|U|^2 < c(r\ell)^2 = \frac{\veps r^2}{128\ell}, 
\]
a contradiction. The proof of \Cref{thm:many_ind_cycle} is complete.

\section{Induced Turán numbers of prism graphs} \label{sec:prism}

This section is devoted to the proof of \Cref{thm:ind_prism}. \Cref{prop:blowup} together with $\ex(n, C_4) = \Omega(n^{3/2})$ gives the lower bound in \Cref{thm:ind_prism}. The upper bound is implied by the following theorem.

\begin{theorem} \label{thm:prism_upper}
     Let $\ell \ge 10$ and $s \ge 2$ be integers. There exists $C = C(\ell) > 0$ such that every $n$-vertex $K_{s, s}$-free graph $G$ with $e(G) \ge C \cdot 6^s s^{20 \ell^2}\cdot n^{3/2}$ contains an induced $C^{\square}_{2\ell}$-subgraph. 
\end{theorem}

To prove \Cref{thm:prism_upper}, we need a supersaturation result on induced $4$-cycles. One can directly deduce such a result using \Cref{thm:many_ind_cycle} and the fact that even cycles satisfy the Sidorenko property. Nonetheless, we give a self-contained proof without relying on \Cref{lem:degenarate-C_2k,lem:red-blue}.

\begin{theorem} \label{thm:C4_supersaturation}
    Let $G$ be an $n$-vertex $K_{s,s}$-free graph. If its number of edges $e(G) > 6^s s \cdot n^{3/2}$, then the number of induced copies of $4$-cycles in $G$ is at least $\frac{e(G)^4}{16n^4}$. 
\end{theorem}

We shall first give the proof of \Cref{thm:prism_upper} assuming \Cref{thm:C4_supersaturation} through \Cref{sec:prism_upper,sec:C4_thick,sec:C4_thin,sec:thin_subtle}, and then establish \Cref{thm:C4_supersaturation} in \Cref{sec:C4_supersaturation}. 

\subsection{Proof sketch of \texorpdfstring{\Cref{thm:prism_upper}}{Theorem 5.1} assuming \texorpdfstring{\Cref{thm:C4_supersaturation}}{Theorem 5.2}} \label{sec:prism_upper}

We use a similar strategy as in the proof of \cite[Theorem 1.1]{gao_janzer_liu_xu}. That is, we separate induced $4$-cycles in $G$ as thick ones and thin ones. Then \Cref{thm:C4_supersaturation} tells that there are either many thick induced $4$-cycles or many thin induced $4$-cycles, and we can find induced $C_{2\ell}^{\square}$ separately in these cases. 

Thanks to \Cref{lem:almost_reg}, we may assume that $G$ is $K$-almost-regular for $K \eqdef 2^{10}$. Suppose $C$ is sufficiently large in terms of $\ell$ and $n$ goes to infinity. (Specifically, $C > \ell^{100\ell}$.) Write $f(s) \eqdef 6^s s^{20 \ell^2}$. Then $G$ is $K$-almost-regular and $K_{s, s}$-free with $e(G) \ge C f(s) n^{3/2}$ and $d \eqdef 2 e(G)/n \ge 2C f(s) n^{1/2}$. 

Assume for the sake of contradiction that $G$ contains no induced copy of $C^{\square}_{2\ell}$. 

Write $g(s) \eqdef f(s)^{1/\ell} = 6^{s/\ell}s^{20\ell}$. Call an induced $4$-cycle $xyzw$ in $G$ \emph{thin} if the codegrees of both diagonal pairs $\deg(x, z), \, \deg(y, w)$ are upper bounded by $C g(s) d^{2/3}$, and \emph{thick} otherwise. \Cref{thm:C4_supersaturation} implies that the number of induced $4$-cycles in $G$ is lower bounded by $\frac{e(G)^4}{16n^4} = 2\veps d^4$, where $\veps \eqdef \frac{1}{512}$. So, there are at least $\veps d^4$ thick induced $4$-cycles or at least $\veps d^4$ thin induced $4$-cycles in $G$. 
\vspace{-0.5em}
\begin{itemize}
    \item We deal with the former case in \Cref{sec:C4_thick}. Roughly, we are going to find an induced copy of $C_{2\ell}^{\square}$ in a dense subgraph of $G$ using a dependent random choice argument. 
    \vspace{-0.5em}
    \item We deal with the latter case in \Cref{sec:C4_thin}. We study an auxiliary graph $\Gamma$ whose vertices and edges refer to edges of $G$ and thin induced $4$-cycles in $G$. Then each nondegenerate $2\ell$-cycle in $\Gamma$ corresponds to a $C_{2\ell}^{\square}$ in $G$. We are going to show by contradiction that there exists a \emph{nice} $2\ell$-cycle in $\Gamma$ which gives an induced copy of $C_{2\ell}^{\square}$ in $G$.
\end{itemize}

We are going to fill in the details of the proof of \Cref{thm:prism_upper} through \Cref{sec:C4_thick,sec:C4_thin,sec:thin_subtle}. From now on, we call a $4$-cycle \emph{thick} (resp.~\emph{thin}) if it is a thick (resp.~thin) induced $4$-cycle in $G$. 

\subsection{There are many thick induced \texorpdfstring{$4$}{4}-cycles} \label{sec:C4_thick}

Assume that the number of thick induced $4$-cycles in $G$ is at least $\veps d^4$. 

Recall that our prism $C_{2\ell}^{\square}$ is a $3$-regular bipartite graph, where each part contains $2\ell$ vertices. So, by \cite[Proposition 2.4 and Lemma 2.3]{hunter_milojevic_sudakov_tomon} applied to $k = 3$ and $H = C_{2\ell}^{\square}$, we deduce the following.

\begin{lemma} \label{lem:rich_to_prism}
    If a subset $X \subseteq V(G)$ with $|X| \ge (16 \ell^2 s)^{8\ell+10}$ has the property that $\deg(S) \ge (16 \ell^2 s)^{8\ell}$ holds for each of its $3$-subsets $S \in \binom{X}{3}$, then $G$ contains an induced copy of $C_{2\ell}^{\square}$. 
\end{lemma}

Notice that $\deg(S)$ denotes the codegree of $S$ in $G$ rather than $X$ in \Cref{lem:rich_to_prism}. 

The pigeonhole principle implies that we can find an edge $xy \in E(G)$ with
\begin{equation} \label{eq:zw_count}
    \bigl| \bigl\{ (z, w) \in V(G)^2 : \text{$xyzw$ is thick with $\deg(x, z) \ge C g(s) d^{2/3}$} \bigr\} \bigr| \ge \frac{\veps d^4 /2}{2e(G)} > C f(s) d. 
\end{equation}
The last inequality follows from the facts $e(G) \ge C f(s) n^{3/2}$ and $d = 2e(G)/n, \, \veps = \frac{1}{512}$. 

Let $A \eqdef N(x)$ and $B \eqdef \bigl\{ z \in N(y) : \deg(x, z) \ge C g(s) d^{2/3} \bigr\}$. Since $G$ is $K$-almost-regular, 
\begin{equation} \label{eq:A_bound}
    |A| = |N(x)| \le Kd. 
\end{equation}
Denote by $e(A, B)$ the number of edges between $A$ and $B$. We derive two lower bounds on $e(A, B)$. 
\vspace{-0.5em}
\begin{itemize}
    \item On one hand, each thick $4$-cycle $xyzw$ corresponds to an edge between $A$ and $B$. Thus, from \eqref{eq:zw_count} together with \eqref{eq:A_bound} we deduce that
    \begin{equation} \label{eq:eAB_fbound}
        e(A, B) \ge C f(s) d \ge \frac{Cf(s)}{K}|A|. 
    \end{equation}

    \vspace{-1em}
    \item On the other hand, the definition of $B$ and \eqref{eq:A_bound} imply that
    \begin{equation} \label{eq:eAB_gbound}
        e(A, B) \ge C g(s) d^{2/3} |B| \ge \frac{C g(s)}{K^{2/3}} |A|^{2/3} |B|. 
    \end{equation}
\end{itemize}
\vspace{-0.5em}

Write $C_1 \eqdef (16 \ell^2 s)^{8\ell+10}$. It follows from $C \gg \ell$ and \eqref{eq:eAB_fbound} that
\[
p \eqdef (2C_1) \cdot \frac{|A|}{e(A, B)} \le (2C_1) \cdot \frac{K}{Cf(s)} < 1. 
\]
Construct a subset $B' \subseteq B$ by including each $b \in B$ independently with probability $p$. Uniformly at random, we pick a vertex $v \in A$ and define $B'' \eqdef B' \cap N(v)$. It follows that 
\begin{align*}
    \E(|B''|) = \sum_{b \in B} p \cdot \frac{\deg_A(b)}{|A|} = p \cdot \frac{e(A, B)}{|A|} = 2C_1. 
\end{align*}
Denote by $\cB$ the collection of $3$-subsets of $B''$ who has at most $C_1$ common neighbors. Then
\begin{align*}
    \E(|\cB|) \le \binom{|B|}{3} \cdot p^3 \cdot \frac{C_1}{|A|} \le \frac{|B|^3}{6} \cdot \frac{8C_1^3|A|^3}{e(A, B)^3} \cdot \frac{C_1}{|A|} = \frac{4 C_1^4 |A|^2 |B|^3}{3 e(A, B)^3} \overset{(*)}{\le} \frac{4 C_1^4 K^2}{3 C^3 g(s)^3} \overset{(**)}{\le} C_1, 
\end{align*}
where we applied \eqref{eq:eAB_gbound} at ($*$) and $C \gg \ell$ at ($**$). So, there is some $(v, B')$ with $|B''| - |\cB| \ge C_1$. For each $3$-set in $\cB$, delete an arbitrary vertex from $B''$ in it to obtain $X$. Then $|X| \ge |B''| - |\cB| \ge C_1$ and each $S \in \binom{X}{3}$ has $\deg(S) \ge C_1$. Therefore, \Cref{lem:rich_to_prism} implies that $G$ contains an induced $C^{\square}_{2\ell}$. 

\subsection{There are many thin induced \texorpdfstring{$4$}{4}-cycles} \label{sec:C4_thin}

Assume that the number of thin induced $4$-cycles in $G$ is at least $\veps d^4$. 

Define an auxiliary graph $\cH$ with $V(\cH) = E(G)$. Put an edge between vertices $e_1 = xy, \, e_2 = zw$ of $\cH$ if and only if $xyzw$ is a thin $4$-cycle in $G$. Then $|V(\cH)| = |E(G)| = nd/2$ and $|E(\cH)|\ge \veps d^4$.

By repeatedly removing vertices of degree less than $\frac{|E(\cH)|}{2|V(\cH)|}$ from the graph $\cH$, we know that $\cH$ contains a subgraph $\Gamma$ such that $|V(\Gamma)| \le nd/2$ and $\delta(\Gamma) \ge \frac{|E(\cH)|}{2|V(\cH)|} \ge d$, as $C \gg \ell$. 

For $e = xy, \, e' = x'y'$, write $e \sim e'$ if and only if $\{x, y\} \cap \{x', y'\} \ne \varnothing$. (Notice that $xy$ and $\{x, y\}$ refer to the same $2$-element set.) Then $\sim$ is a symmetric binary relation over $V(\Gamma)$. For any pair of (not necessarily distinct) vertices $e_1, e_2 \in V(\Gamma)$, the definition of thin $4$-cycles implies that 
\begin{equation} \label{eq:sim_edge}
    \bigl| \bigl\{ e_3 \in V(\Gamma) : e_1 \sim e_3, \, e_2e_3 \in E(\Gamma) \bigr\} \bigr| \le 4C g(s) d^{2/3}. 
\end{equation}

We are going to find a subgraph of $\Gamma$ containing many homomorphic copies of $C_{2\ell}$. To achieve this, we use Sidorenko property of $C_{2\ell}$ and a regularization result of Janzer \cite{janzer_disproof}. 

\begin{lemma}[{\cite[Lemma 2.3]{janzer_disproof}}] \label{lem:bi_regular}
    Let $G$ be an $n$-vertex graph with average degree $d > 0$. Then there exist $D_1, D_2 \ge d/4$ and a non-empty bipartite subgraph $G'$ of $G$ with parts $X, Y$ such that 
    \vspace{-0.5em}
    \begin{itemize}
        \item for every $x \in X$, we have $\deg_{G'}(x) \ge \frac{D_1}{256(\log n)^2}$ and $\deg_G(x) \le D_1$, and 
        \vspace{-0.5em}
        \item for every $y \in Y$, we have $\deg_{G'}(y) \ge \frac{D_2}{256(\log n)^2}$ and $\deg_G(y) \le D_2$. 
    \end{itemize}
    \vspace{-0.5em}
\end{lemma}

\begin{lemma}[{\cite[Lemma 2.4]{janzer_disproof}}] \label{lem:hom_C2l}
    Suppose $k \in \N_+$ and let $G$ be a bipartite graph with parts $X, Y$ such that $\deg(x) \ge a \, (\forall x \in X)$ and $\deg(y) \ge b \, (\forall y \in Y)$. Then $\hom(C_{2k}, G) \ge a^k b^k$. 
\end{lemma}

According to \Cref{lem:bi_regular}, there is a non-empty bipartite subgraph $\Gamma'$ of $\Gamma$ with parts $X, Y$ with
\vspace{-0.5em}
\begin{itemize}
    \item for every $x \in X$, we have $\deg_{\Gamma'}(x) \ge \frac{D_1}{256(\log  \lvert V(\Gamma) \rvert)^2}$ and $\deg_{\Gamma}(x) \le D_1$, and 
    \vspace{-0.5em}
    \item for every $y \in Y$, we have $\deg_{\Gamma'}(y) \ge \frac{D_2}{256(\log \lvert V(\Gamma) \rvert)^2}$ and $\deg_{\Gamma}(y) \le D_2$, 
\end{itemize}
\vspace{-0.5em}
where $D_1, D_2 \ge \delta(\Gamma)/4 \ge d/4$. It then follows from \Cref{lem:hom_C2l} and $\ell \ge 10, \, d \to \infty$ that
\begin{equation} \label{eq:hom_C2l}
    \hom(C_{2\ell}, \Gamma') \ge \biggl( \frac{D_1D_2}{256^2 ( \log \lvert V(\Gamma) \rvert )^4} \biggr)^{\ell} \ge d^{19}. 
\end{equation}

\medskip

Assume for the sake of contradiction that $G$ contains no induced copy of $C_{2\ell}^{\square}$. Observe that any homomorphic $2\ell$-cycle $(e_1, \dots, e_{2\ell})$ in $\Gamma$ correspondes to a homomorphic copy of $C_{2\ell}^{\square}$ in $G$. We call a homomorphic $2\ell$-cycle $(e_1, \dots, e_{2\ell})$ in $\Gamma$ \emph{degenerate} if $e_i \sim e_j$ for some $i \ne j$. Since $G$ contains no induced $C_{2\ell}^{\square}$, every non-degenerate homomorphic $(e_1, \dots, e_{2\ell})$ fits into exactly one of the followings. 
\vspace{-0.5em}
\begin{itemize}
    \item If there is an edge in $G$ between non-adjacent $e_i, e_j$ on $(e_1, \dots, e_{2\ell})$ for some $i \notin \{1, \ell+1\}$ and $j \notin \{1, \ell+1\}$, then $(e_1, \dots, e_{2\ell})$ is \emph{typical}. Otherwise, $(e_1, \dots, e_{2\ell})$ is \emph{special}. 
\end{itemize}
\vspace{-0.5em}
So, the set $\Hom(C_{2\ell}, \Gamma')$ is partitioned into \emph{degenerate}, \emph{typical}, or \emph{special} homomorphic $2\ell$-cycles. 

\begin{claim} \label{claim:special_homo}
    In $\Hom(C_{2\ell}, \Gamma')$, the proportion of special ones is at most $2/\ell$. 
\end{claim}

\begin{poc}
    Let $(e_1, \dots, e_{2\ell}) \in \Hom(C_{2\ell}, \Gamma')$ be a homomorphic $2\ell$-cycle. Among its $2\ell$ rotations
    \[
    (e_1, e_2, \dots, e_{2\ell}), \qquad (e_2, e_3, \dots, e_1), \qquad \dots, \qquad (e_{2\ell}, e_1, \dots, e_{2\ell-1}), 
    \]
    at most $4$ are special. This implies that the proportion of special ones is at most $4/(2\ell) = 2/\ell$. 
\end{poc}

Let $\cV \eqdef V(\Gamma')^2$ be the set of (not necessarily distinct) vertex pairs in $\Gamma'$. For $(x, y) \in \cV$, define
\begin{align*}
    \cW_{x, y} &\eqdef \bigl\{ (e_1, \dots, e_{\ell+1}) \in \Hom(P_{\ell}, \Gamma') : e_1 = x, \, e_{\ell+1} = y \bigr\}, \\
    \cC_{x, y} &\eqdef \bigl\{ (e_1, \dots, e_{2\ell}) \in \Hom(C_{2\ell}, \Gamma') : e_1 = x, \, e_{\ell+1} = y \bigr\}. 
\end{align*}
One can interpret $\cW_{x, y}$ as the collection of $\ell$-walks between $x$ and $y$. Clearly, we have
\[
|\cW_{x, y}|^2 = |\cC_{x, y}|, \qquad \lvert \Hom(C_{2\ell}, \Gamma') \rvert = \sum_{(x, y) \in \cV} |\cC_{x, y}|. 
\]

Fix parameters $\lambda \eqdef s \cdot 80^{s} \cdot \ell^{2s}$ and $\delta \eqdef \frac{1}{100\lambda^2}$. 

\begin{claim} \label{claim:degenerate_homo}
    In $\Hom(C_{2\ell}, \Gamma')$, the proportion of degenerate ones is at most $\delta$. 
\end{claim}

\begin{claim} \label{claim:rich_walk}
    If $|\cW_{x, y}| > \lambda$, then in $\cC_{x, y}$ we have that
    \vspace{-0.5em}
    \begin{itemize}
        \item the proportion of degenerate homomorphic $2\ell$-cycles is at least $4\delta$, or 
        \vspace{-0.5em}
        \item the proportion of special homomorphic $2\ell$-cycles is at least $1/2$. 
    \end{itemize}
    \vspace{-0.5em}
\end{claim}

Due to subtlety, the proofs of \Cref{claim:degenerate_homo,claim:rich_walk} are postponed to \Cref{sec:thin_subtle}. 

\medskip

Observe that \Cref{claim:degenerate_homo,claim:special_homo} provide upper bounds while \Cref{claim:rich_walk} gives lower bounds on homomorphic $2\ell$-cycles. By summing $|\cC(x, y)|$ over $\cV$, we are going to derive a contradiction on the quantity $\hom(C_{2\ell}, \Gamma')$. Consider the three subsets $\cV_1, \cV_2, \cV_3$ of $\cV$, where
\begin{align*}
    \cV_1 &\eqdef \bigl\{ (x, y) \in \cV : |\cW_{x, y}| \le \lambda \bigr\} = \bigl\{ (x, y) \in \cV : |\cC_{x, y}| \le \lambda^2 \bigr\}, \\
    \cV_2 &\eqdef \bigl\{ (x, y) \in \cV : \text{at least $4\delta$ proportion of $\cC_{x, y}$ is degenerate} \bigr\}, \\
    \cV_3 &\eqdef \bigl\{ (x, y) \in \cV : \text{at least $1/2$ proportion of $\cC_{x, y}$ is special} \bigr\}. 
\end{align*}

Then \Cref{claim:rich_walk} tells us that $\cV = \cV_1 \cup \cV_2 \cup \cV_3$. (This is not necessarily a partition since $\cV_2, \cV_3$ may intersect.) Recall that $|\cV| \le |V(\Gamma)|^2 < n^2d^2$. From the definition of $\cV_1$ and \eqref{eq:hom_C2l} we deduce that 
\begin{equation} \label{eq:V1_bound}
    \sum_{(x, y) \in \cV_1} |\cC_{x, y}| \le \lambda^2 |\cV_1| \le \lambda^2 |\cV| < \frac{1}{4} \hom(C_{2\ell}, \Gamma'). 
\end{equation}
Counting degenerate homomorphic $2\ell$-cycles, the definition of $\cV_2$ and \Cref{claim:degenerate_homo} imply that 
\begin{equation} \label{eq:V2_bound}
   \delta \cdot \hom (C_{2\ell}, \Gamma') \ge \sum_{(x, y) \in \cV_2} 4\delta \cdot |\cC(x, y)|. 
\end{equation}
Counting special homomorphic $2\ell$-cycles, the definition of $\cV_3$ and \Cref{claim:special_homo} imply that 
\begin{equation} \label{eq:V3_bound}
    \frac{2}{\ell} \cdot \hom (C_{2\ell}, \Gamma') \ge \sum_{(x, y) \in \cV_3} \frac{1}{2} \cdot |\cC(x, y)|. 
\end{equation}
Since $\cV = \cV_1 \cup \cV_2 \cup \cV_3$, combining \eqref{eq:V1_bound}, \eqref{eq:V2_bound}, and \eqref{eq:V3_bound} we obtain 
\begin{align*}
    \hom(C_{2\ell}, \Gamma') = \sum_{(x, y) \in \cV} |\cC_{x, y}| &\le \sum_{(x, y) \in \cV_1} |\cC_{x, y}| + \sum_{(x, y) \in \cV_2} |\cC_{x, y}| + \sum_{(x, y) \in \cV_3} |\cC_{x, y}| \\
    &< \Bigl( \frac{1}{4} + \frac{1}{4} + \frac{4}{\ell} \Bigr) \cdot \hom(C_{2\ell}, \Gamma'), 
\end{align*}
which is the desired contradiction because $\ell \ge 10$. 

\subsection{Deferred technical proofs} \label{sec:thin_subtle}

To prove \Cref{claim:degenerate_homo}, we closely follow the proof of \cite[Lemma 2.5]{janzer_disproof}. Nevertheless, we include it here for completeness. We need a technical result introduced by Janzer \cite{janzer_disproof}.

\begin{lemma}[{\cite[Lemma 2.2]{janzer_disproof}}] \label{lem:degenerate_cycle}
    Let $\ell \ge 2$ be an integer. Suppose $G = (V, E)$ is an $n$-vertex graph. Let $X_1, X_2$ be subsets of $V$, and $\sim$ be a symmetric binary relation defined over $V$ such that 
    \vspace{-0.5em}
    \begin{itemize}
        \item for every $u \in V$ and $v \in X_1$, the vertex $v$ has at most $\Delta_1$ neighbours $w \in X_2$ and amongst them at most $s_1$ satisfies $u \sim w$, and
        \vspace{-0.5em}
        \item for every $u \in V$ and $v \in X_2$, the vertex $v$ has at most $\Delta_2$ neighbours $w \in X_1$ and amongst them at most $s_2$ satisfies $u \sim w$. 
    \end{itemize}
    \vspace{-0.5em}
    Write $M \eqdef \max\{ \Delta_1s_2, \Delta_2s_1\}$. Then the number of homomorphic $2\ell$-cycles 
    \[
    (x_1,x_2,\dots ,x_{2\ell}) \in (X_1 \times X_2 \times X_1 \times \dots \times X_2) \cup (X_2 \times X_1 \times X_2 \times \dots \times X_1)
    \]
    in $G$ satisfying $x_i \sim x_j$ for some $i \ne j$ is at most $32 \ell^{3/2}M^{1/2} n^{\frac{1}{2\ell}} \hom(C_{2\ell}, G)^{1-\frac{1}{2\ell}}$. 
\end{lemma}

\begin{proof}[Proof of \Cref{claim:degenerate_homo}]
    Set $m \eqdef |V(\Gamma')|$ and $\alpha \eqdef 16 C g(s) d^{-1/3}$. Then $m \le |V(\Gamma)| \le nd/2$. 
    
    Since $D_1 \ge d/4$ and $D_2 \ge d/4$, the definition of thin $4$-cycles implies that
    \vspace{-0.5em}
    \begin{itemize}
        \item for every $u \in V(\Gamma')$ and $v \in X$, we have $\bigl|\bigl\{ w \in N_{\Gamma'}(v) : u \sim w \bigr\}\bigr| \le 4 C g(s) d^{2/3} \le \alpha D_1$, and 
        \vspace{-0.5em}
        \item for every $u \in V(\Gamma')$ and $v \in Y$, we have $\bigl|\bigl\{ w \in N_{\Gamma'}(v) : u \sim w \bigr\}\bigr| \le 4 C g(s) d^{2/3} \le \alpha D_2$. 
    \end{itemize}
    \vspace{-0.5em}
    Applying \Cref{lem:degenerate_cycle} to $\Gamma'$ with parameters $(n, \Delta_1, \Delta_2, s_1, s_2, M) = (m, D_1, D_2, \alpha D_1, \alpha D_2, \alpha D_1 D_2)$, we deduce that the number of degenerate homomorphic $2\ell$-cycles in $\Gamma'$ is at most
    \begin{equation} \label{eq:degenerate_janzer}
        32 \ell^{3/2} (\alpha D_1 D_2)^{1/2} m^{\frac{1}{2\ell}} \hom(C_{2\ell}, \Gamma')^{1-\frac{1}{2\ell}}. 
    \end{equation}

    As $d \to \infty$, it follows from $m \le nd$ and $\ell \ge 10$ that 
    \begin{align} \label{eq:alpha_bound}
        m^{1/\ell} (\log m)^4 \ll d^{1/3} &\implies 2^{26} \ell^3 \alpha m^{1/\ell} (\log m)^4 \le \delta^2 \nonumber \\
        &\implies 2^{10} \ell^3 (\alpha D_1 D_2) m^{1/\ell} \le \frac{\delta^2 D_1 D_2}{2^{16} (\log m)^4} \nonumber \\
        &\implies 32 \ell^{3/2} (\alpha D_1 D_2)^{1/2} m^{\frac{1}{2\ell}} \le \delta \biggl( \frac{D_1 D_2}{256^2 (\log m)^4} \biggr)^{1/2}. 
    \end{align}
    Since \eqref{eq:hom_C2l} tells us that $\hom(C_{2\ell}, \Gamma') \ge \big(\frac{D_1D_2}{256^2 (\log m)^4}\big)^{\ell}$, combining \eqref{eq:degenerate_janzer} and \eqref{eq:alpha_bound} we conclude that the number of degenerate homomorphic $2\ell$-cycles in $\Gamma'$ is upper bounded by $\delta \hom(C_{2\ell}, \Gamma')$. 
\end{proof}

Notice that the proof of \Cref{claim:degenerate_homo} is the place where we cannot break the $\ell \ge 10$ barrier. 

\begin{proof}[Proof of \Cref{claim:rich_walk}]
    For $\ell$-walk $W = (x, e_1, \dots, e_{\ell-1}, y) \in \cW_{x, y}$, view $e_i$ as a $2$-vertex set in $V(G)$ and let $V^*_G(W) \eqdef \bigcup_{i=1}^{\ell-1} e_i$ be the vertices that the interior of $W$ corresponds to in $G$. For two $\ell$-walks 
    \[
    W = (x, e_1, \dots, e_{\ell-1}, y), \qquad W' = (x, e_1', \dots, e_{\ell-1}', y)
    \]
    let $C(W, W') \eqdef (x, e_1, \dots, e_{\ell-1}, y, e_{\ell-1}', \dots, e_1') \in \cC_{x, y}$ be the  homomorphic $2\ell$-cycle  in $\Gamma'$ obtained by concatenating $W$ and $W'$. For $\lambda$-element subset $\{W_1, \dots, W_{\lambda}\}$ of $\cW_{x, y}$, call it \emph{good} if $C(W_i, W_j)$ is degenerate for some $i \ne j$, and \emph{bad} otherwise (where $V^*_G(W_1), \dots, V^*_G(W_{\lambda})$ are pairwise disjoint). 

    \begin{claim} \label{claim:bad_walk}
        If $\{W_1, \dots, W_{\lambda}\}$ is bad, then $\bigl|\bigl\{ \text{$C(W_i, W_j)$ \emph{is special}} : i, j \in [\lambda] \bigr\}\bigr| \ge \frac{4\lambda^2}{5}$. 
    \end{claim}

    \begin{poc}
        In the graph $G$, we consider the induced subgraph $H \eqdef G \bigl[ V^*_G(W_1) \cup \dots \cup V^*_G(W_{\lambda}) \bigr]$. Recall that $\{W_1, \dots, W_{\lambda}\}$ is bad hence the vertices are pairwise distinct. Then $|V(H)| = 2\lambda(\ell-1)$. If $\bigl|\bigl\{ \text{$C(W_i, W_j)$ is typical} : i, j \in [\lambda] \bigr\}\bigr| \ge \frac{\lambda^2}{10}$, then pairing up $C(W_i, W_j)$ and $C(W_j, W_i)$, we have
        \[
        |E(H)| \ge \frac{\lambda^2}{20} > \frac{|V(H)|^2}{80\ell^2}. 
        \]
        Here we are counting the edges between $V^*_G(W_i)$ and $V^*_G(W_j)$ in $G$. However, since $G$ (and hence $H$) is $K_{s, s}$-free, \Cref{coro:KST} applied to $c = \frac{1}{80\ell^2}$ implies that $|E(H)| \le \frac{|V(H)|^2}{80\ell^2}$, a contradiction. So, 
        \[
        \bigl|\bigl\{ \text{$C(W_i, W_j)$ is special} \bigr\}\bigr| = \lambda(\lambda-1) - \bigl|\bigl\{ \text{$C(W_i, W_j)$ is typical} \bigr\}\bigr| \ge \frac{4\lambda^2}{5}. \qedhere
        \]
    \end{poc}
    
    Write $w \eqdef |\cW_{x, y}|$. There are $\binom{w}{\lambda}$ many $\lambda$-sets of $\ell$-walks in $\cW_{x, y}$. 
    \vspace{-0.5em}
    \begin{itemize}
        \item If at least $1/4$ of such $\lambda$-sets are good, then from $\delta = \frac{1}{100\lambda^2}$ we deduce that at least
        \[
        \frac{1}{4} \cdot \frac{\binom{w}{\lambda}}{\binom{w-2}{\lambda-2}} = \frac{1}{4} \cdot \frac{w(w-1)}{\lambda(\lambda-1)} \ge 4\delta w^2 = 4\delta \cdot |\cW_{x, y}|^2 = 4\delta \cdot |\cC_{x, y}|
        \]
        many homomorphic $2\ell$-cycles in $\cC_{x, y}$ are degenerate. 
        \vspace{-0.5em}
        \item If at least $3/4$ of such $\lambda$-sets are bad, then \Cref{claim:bad_walk} implies that at least
        \[
        \frac{3}{4} \cdot \frac{4\lambda^2}{5} \cdot \frac{\binom{w}{\lambda}}{\binom{w-2}{\lambda-2}} = \frac{3\lambda^2}{5} \cdot \frac{w(w-1)}{\lambda(\lambda-1)} \ge \frac{w^2}{2} = \frac{1}{2} \cdot |\cW_{x, y}|^2 = \frac{1}{2} \cdot |\cC_{x, y}|
        \]
        many homomorphic $2\ell$-cycles in $\cC_{x, y}$ are special. 
    \end{itemize}
    \vspace{-0.5em}
    Combining the cases above, the proof of \Cref{claim:rich_walk} is complete. 
\end{proof}

\subsection{Proof of \texorpdfstring{\Cref{thm:C4_supersaturation}}{Theorem 5.2}} \label{sec:C4_supersaturation}

For any vertex $v \in V(G)$, we denote 
\[
\cP_{*v*} \eqdef \bigl\{ \text{induced copies of $2$-paths $xvy$ in $G$ for some $x, y \in V(G)$} \bigr\}. 
\]
Notice that $xvz$ and $zvx$ are regarded as a same copy. If $\deg(v) \ge 6^s s$, then \Cref{coro:KST} implies that the number of edges in $N(v)$ is upper bounded by $\frac{\deg(v)^2}{6}$. Observe that each non-edge in $N(v)$ corresponds to a unique path in $\cP_{*v*}$. This implies that 
\[
|\cP_{*v*}| \ge \binom{\deg(v)}{2} - \frac{\deg(v)^2}{6} \ge \frac{\deg(v)^2}{4}. 
\]
Denote by $\mathds{1}_{E}$ the indicator of event $E$. Then in general, we have that for any $v\in V(G)$,
\begin{equation} \label{eq:ind_path_v}
|\cP_{*v*}| \ge \frac{\deg(v)^2}{4} \cdot \mathds{1}_{\{\deg(v) \ge 6^s s\}} + 0 \cdot \mathds{1}_{\{\deg(v) < 6^s s\}} \ge \frac{\deg(v)^2}{4} - 6^{2s}s^2. 
\end{equation}
Summing over all the vertices in $G$, by Cauchy--Schwarz, we obtain
\begin{align} \label{eq:Pv}
    \sum_{v \in V(G)} |\cP_{*v*}| &\ge \frac{1}{4} \sum_{v \in V(G)} \deg(v)^2 - 6^{2s} s^2 n \ge \frac{n}{4} \biggl( \frac{1}{n} \sum_{v \in V(G)} \deg(v) \biggr)^2 - 6^{2s} s^2 n \nonumber \\
    &\ge \frac{e(G)^2}{n}- 6^{2s} s^2 n \ge \frac{3e(G)^2}{4n}. 
\end{align}

For any pair of distinct vertices $u, v \in V(G)$, we define
\begin{align*}
    \cP_{u*v} &\eqdef \bigl\{ \text{induced copies of $2$-paths $uxv$ in $G$ for some $x \in V(G)$} \bigr\}, \\
    \cC_{u*v*} &\eqdef \bigl\{ \text{induced copies of $4$-cycles $uxvy$ in $G$ for some $x, y \in V(G)$} \bigr\}. 
\end{align*}
Notice that $uxv, vxu$ are regarded as a same $2$-path, and $uxvy, \, uyvx, \, vxuy, \, vyux$ are regarded as a same $4$-cycle. Note that each pair of paths in $\cP_{u*v}$ whose midpoints are not adjacent corresponds to an induced $4$-cycle in $\cC_{u*v*}$. Thus, similar to \eqref{eq:ind_path_v}, we may deduce from \Cref{coro:KST} that 
\[
|\cC_{u*v*}| \ge \frac{|\cP_{u*v}|^2}{4} - 6^{2s}s^2. 
\]
Observe that each induced $4$-cycle (e.g., $uxvy$) is contained in both $\cC_{u*v*}$ and $\cC_{x*y*}$. It follows that the total number of induced copies of $4$-cycles in $G$ is
\begin{align} \label{eq:Cuv}
    \frac{1}{2} \sum_{\{u,v\} \in \binom{V(G)}{2}} |\cC_{u*v*}| &\ge \frac{1}{2} \Biggl( \sum_{\{u,v\} \in \binom{V(G)}{2}} \frac{|\cP_{u*v}|^2}{4} - 6^{2s} s^2 {\binom{n}{2}} \Biggr) \nonumber \\
    &\ge \frac{1}{4n^2} \Biggl( \sum_{\{u,v\} \in \binom{V(G)}{2}} |\cP_{u*v}| \Biggr)^2 - \frac{1}{2} \cdot 6^{2s} s^2 n^2. 
\end{align}
Notice that $\sum\limits_{v \in V(G)} |\cP_{*v*}| = \sum\limits_{\{u,v\} \in \binom{V(G)}{2}} |\cP_{u*v}|$, because both of the sums count the total number of induced copies of $2$-paths in $G$. We thus deduce from \eqref{eq:Pv} and \eqref{eq:Cuv} that there are at least
\[
\frac{1}{4n^2} \cdot \biggl( \frac{3e(G)^2}{4n} \biggr)^2 - \frac{1}{2} \cdot 6^{2s} s^2 n^2 \ge \frac{e(G)^4}{16n^4}
\]
many induced copies of $4$-cycles in $G$. The proof of \Cref{thm:C4_supersaturation} is complete. 

\section{Concluding remarks} \label{sec:remark}

%In \cite{hunter_milojevic_sudakov_tomon}, Hunter et al.~proposed the framework which unifies the study of Tur\'{a}n-type problems with the study of induced subgraphs. In this paper, we continued this line of research and proved that $\ex^*(n, \cH, s), \, \ex(n, \cH)$ have the same asymptotic behavior when $\cH$ is some specific family of balanced rooted trees, a large theta graph, and a large prism graph. We believe that every finite family of fixed bipartite graphs should exhibit the same phenomenon (\Cref{conj:HMST+}). 

\Cref{thm:ind_theta,thm:ind_prism} give correct asymptotics of induced Tur\'{a}n numbers on $n$. However, the dependencies on $s$ we obtain are relatively poor, as the upper bounds are exponential while the lower bounds are polynomial. In particular, these upper bounds on $\ex^*(n, H, s)$ in terms of $s$ grow with the order of $H$ yet the lower bounds do not. We believe that $\ex^*(n, \Theta_{\ell}^t, s)$ should grow with $t$, and leave the improvement of \Cref{prop:blowup} as an open problem. 
   
The approach of Gao et al.~\cite{gao_janzer_liu_xu} finds a copy of $C_{2\ell}^{\square}$ in every graph $G$ with $e(G) \ge C n^{3/2}$ for all $\ell \ge 4$. However, their method does not seem to be easily adapted to find an induced copy. In our argument, we cannot go below the $\ell = 10$ barrier because of the proof of \Cref{claim:degenerate_homo}. An interesting open problem is to come up with different arguments and establish the same result for smaller $\ell$. We believe that the $O(n^{3/2})$ upper bound holds for all $\ell\ge 3$.

\bibliographystyle{plain}
\bibliography{ind_bip_turan}

\end{document}